\RequirePackage[l2tabu, orthodox]{nag}

\documentclass[reqno,10pt]{amsart}

\usepackage{graphicx}
\usepackage{hyperref}

\newcommand{\Prob}{\mathbb{P}}
\newcommand{\E}{\mathbb{E}}
\newcommand{\C}{\mathbb{C}}

\newcommand{\R}{\mathbb{R}}

\newcommand{\Res}{\mathbf{G}}
\newcommand{\Ker}{\mathcal{K}}

\newcommand{\cC}{\mathcal{C}}

\newcommand{\cL}{\mathcal{L}}

\newcommand{\BA}{\mathbf{A}}
\newcommand{\Id}{\mathbf{I}}

\DeclareMathOperator{\Tr}{Tr}
\DeclareMathOperator{\sgn}{sgn}
\DeclareMathOperator{\diff}{d}

\newtheorem{theorem}{Theorem}
\newtheorem{lemma}[theorem]{Lemma}
\newtheorem{corollary}[theorem]{Corollary}

\theoremstyle{definition}
\newtheorem{definition}[theorem]{Definition}

\newtheorem{remark}[theorem]{Remark}

\title[Covariance of Half-Heavy Wigner Matrix Eigenvalues]{Covariance 
Kernel of Linear Spectral Statistics for Half-Heavy Tailed Wigner 
Matrices}
\author{Asad Lodhia and Anna Maltsev}

\begin{document}

\begin{abstract}
In this paper we analyze the covariance kernel of the Gaussian process 
that arises as the limit of fluctuations of linear spectral statistics 
for Wigner matrices with a few moments. More precisely, the process we 
study here corresponds to Hermitian matrices with independent entries 
that have $\alpha$ moments for $2<\alpha < 4$. We obtain a closed form 
$\alpha$-dependent expression for the covariance of the limiting 
process resulting from fluctuations of the Stieltjes transform by 
explicitly integrating the known double Laplace transform integral 
formula obtained in \cite{bgm16}.  We then express the covariance as an 
integral kernel acting on bounded continuous test functions. The 
resulting formulation allows us to offer a heuristic interpretation of 
the impact the typical large eigenvalues of this matrix ensemble have 
on the covariance structure.
\end{abstract}
\maketitle

\section{Introduction}

The main purpose of this paper is to understand the covariance
of the \emph{linear spectral statistics} for Wigner matrices $\BA_N$
whose entries have cumulative distribution functions decaying like
$x^{-\alpha}$ for $2 < \alpha < 4$ (see Definition \ref{def:matmodel}).
Linear spectral statistics are random variables of the type
\begin{equation*}
X_N(f) := \sum_{j=1}^N f\big(\lambda_j(\BA_N)\big)
\end{equation*}
where  $\{\lambda_j(\BA_N)\}_{j=1}^N$ are the eigenvalues of $\BA_N$,
$f$ is a test function, and $N$ is the dimension of $\BA_N$.

Studying properties of the linear statistic of the eigenvalues allows
us to characterize the limiting behavior of the empirical measure,
\begin{equation*}
L_N(\diff x) = \frac{1}{N}\sum_{j=1}^N \delta_{\lambda_j}(\diff x).
\end{equation*}
The celebrated Wigner semicircle law, the Mar\v{c}enko-Pastur Law and
various other limiting distribution of random matrix eigenvalues are
shown by proving the convergence of $L_N(f) = N^{-1} X_N(f)$ for a
well-chosen class of functions $f$, for instance, the resolvent $f(x) =
(z-x)^{-1}$ for $z \in \C^+$ or polynomial test functions $f(x) =
x^k$ for $k \geq 1$, see \cite{AGZ10} for a detailed introduction.

Once the limiting measure of $L_N$ is understood, a natural way to
proceed is to look at the fluctuations of the linear statistics. This
amounts to studying the quantity
\begin{equation*}
\mathring{X}_N(f):=\sigma_N \Big(X_N(f) - \E\big[X_N(f)\big]\Big),
\end{equation*}
where $f$ is suitably smooth and $\sigma_N$ is an appropriate scaling 
parameter. For many matrix models, choosing the factor $\sigma_N = 1$ 
causes $\mathring{X}_N(f)$ (or even $X_N(f)$ itself) to converge to a 
centered Gaussian random variable with variance functional $\Sigma(f)$.

In addition to establishing the Gaussianity of $\mathring{X}_N(f)$, a 
thorough analysis of its covariance structure is important both as a 
theoretical achievement and as a starting point for the development of 
novel statistical techniques. A deep understanding of such covariance 
structures can establish which sets of test functions will be 
asymptotically independent of others. This information underlies many 
statistical tests arising in applications (see the recent survey 
\cite{YZhB15} for a variety of settings for which Central Limit 
Theorems for linear spectral statistics can be used for theoretical 
statistics.)

\subsection{Central Limit Theorems for Light-Tailed Ensembles}

In the above description, if the matrix $\BA_N$ is a Haar distributed
Unitary matrix (similar results hold for Haar distributed Orthogonal
and Symplectic random matrices), $L_N(\diff z)$ converges to a uniform
measure on the unit circle. The \emph{joint} fluctuations of the vector
\begin{equation*}
\big(X_N(f_1),\ldots,X_N(f_k)\big),
\end{equation*}s
where $f_\ell(w) = w^\ell$ converge to a centered complex Gaussian
random vector whose covariance matrix is
\begin{equation*}
\begin{bmatrix}
1 & 0 & \cdots & \cdots & 0 \\
0 & \sqrt{2} & 0 & \cdots & 0  \\
\vdots & \ddots & \ddots & \ddots & \vdots\\
0 & \cdots & 0 & \sqrt{k-1} & 0 \\
0 & \cdots & \cdots  &0 &\sqrt{k}
\end{bmatrix}
\end{equation*}
which is to say, $X_N$ converges to a random Fourier series whose
coefficients of $f_\ell$ are independent Gaussians with variance $\ell$
\cite[Theorem 2]{ds94}. The proof of \cite{ds94} relied on explicit
formula of the moments in terms of symmetric polynomials due to their
relationship to representation theory of the classical compact groups.

If the matrix $\BA_N$ is a Hermitian random matrix whose density
is of the form
\begin{equation*}
\Prob\big(\{\BA_N \in S\}\big)=\frac{1}{\mathcal{Z}_N}\int_S
\exp\bigg(-\frac{N}{2}\Tr V(\Phi)\bigg) \prod_{1\leq i < j\leq N}
\diff\Re\Phi_{i,j} \diff\Im \Phi_{i,j} \prod_{k=1}^N \diff \Phi_{k,k}
\end{equation*}
for a polynomial potential $V$ with even degree and largest coefficient
positive then $L_N$ converges to an ``equilibrium measure,'' $\mu_V$.
Assuming $\mu_V$  is supported on an interval $[a,b]$, for Chebyshev
polynomials $T_1$, $\ldots$, $T_k$, the multivariate vector
\begin{equation*}
\big( X_N(T_1) - N\mu_V(T_1), \ldots\/, X_N(T_k) - N\mu_V(T_k)\big),
\end{equation*}
converges to a real multivariate Gaussian with covariance matrix
\begin{equation*}
\frac{1}{2}
\begin{bmatrix}
1 & 0 & \cdots & \cdots & 0 \\
0 & \sqrt{2} & 0 & \cdots & 0  \\
\vdots & \ddots & \ddots & \ddots & \vdots\\
0 & \cdots & 0 & \sqrt{k-1} & 0 \\
0 & \cdots & \cdots  &0 &\sqrt{k}
\end{bmatrix},
\end{equation*}
again. We may interpret the above result as saying that $X_N - N\mu_V$ 
converges as a process to a random Fourier series in the basis $T_j$ 
with similar results for real symmetric and symplectic random matrices 
\cite{J98}. The proof of \cite{J98} relied on potential theory 
arguments and integration by parts formulae for the explicit density of 
the eigenvalues.

The above Central Limit Theorems hold true for other point processes
$\{\lambda_j\}_{j=1}^N \subset \R$ even when they do not have an
immediate interpretation as the eigenvalues of some random matrix
model. For example, when $\{\lambda_j\}_{j=1}^N$ are distributed
according to a $\beta$-ensemble
\begin{equation*}
\Prob\big(\{(\lambda_1,\ldots,\lambda_N) \in S\}\big) =
\frac{1}{\mathcal{Z}_{N,\beta}}\int_S \prod_{1 \leq i < j \leq
N}|\lambda_i - \lambda_j|^\beta \exp\bigg(-\frac{\beta N}{2}\sum_{i=1}^N
V(\lambda_i)\bigg),
\end{equation*}
for a wide range of $V$ and $\beta$, the
Central Limit Theorem proved in \cite{J98} still holds.
The interpretation of this point process as a set of random matrix
eigenvalues is lost when $\beta \notin\{1,2,4\}$.

When $\{\lambda_j\}_{j=1}^N\subset \R$ are distributed according to 
biorthogonal ensembles \cite[Theorem 2.5]{BD17} whose joint 
distributions are defined by the formula
\begin{multline*}
\Prob\big(\{(\lambda_1,\ldots, \lambda_N)\in S\}\big)\\ =\int_S
\frac{1}{N!}\det \Big[\psi_{j-1}(\lambda_i)\Big]_{1\leq i,j \leq N}
\det\Big[ \phi_{j-1}(\lambda_i)\Big]_{1\leq i,j \leq N} \nu(\diff
\lambda_1) \cdots \nu(\diff\lambda_N),
\end{multline*}
where $\nu$ is a Borel measure and $\phi_i$ and $\psi_j$ are a family 
of functions on $\R^N$ such that
\begin{equation*}
\int_\R \phi_i(x) \psi_j(x) \diff \nu(x) = \delta_{i,j},
\end{equation*}
then under certain assumptions on $\phi_i$ and $\psi_j$, there exists a
measure $\mu$ supported on a single interval $[a,b]$ such that
\begin{equation*}
L_N(\diff x) \to \mu
\end{equation*}
weakly almost surely. Further for any $f \in C^1(\R)$ it has been
shown that the following weak convergence holds
\begin{equation*}
X_N(f) - N\mu(f)\Longrightarrow \mathcal{N}\bigg(0,\sum_{k=1}^\infty k
|\hat f_k|^2\bigg),
\end{equation*}
where $\hat{f}_k$ are explicit.

In all of the above cases, when the limit shape of our point process 
$L_N$ is not ``multi-cut'' (supported on several disjoint non-empty 
intervals), the fluctuations of $X_N$ is a random Fourier series whose 
$k$-th coefficient is a Gaussian times $\sqrt{k}$ (up to some 
model-dependent scaling and centering). Non-Gaussian behavior has 
arisen for the $\beta$-ensemble when the limiting measure $\mu_V$ is 
multi-cut \cite{Sc13}, therefore while Gaussianity describes 
fluctuations of linear spectral statistics for some of the most common 
random matrix ensembles, this characterization of the fluctuations of 
linear spectral statistics is not all-encompassing.

\subsection{Heavy Tailed Matrices}

The proof techniques used in the above results for classical compact 
group random matrices, light-tailed matrices, general $\beta$-ensembles 
and biorthogonal ensembles no longer apply for heavy-tailed random 
matrix eigenvalues. When the entries $a_{i,j}$ of $\mathbf{A}_N$ are of 
the form $N^{-\frac{1}{\alpha}}x_{i,j}$ with $0 < \alpha < 2$ and 
$x_{i,j}$ in the domain of attraction of an $\alpha$-stable law, the 
limit of $L_N$ is now a measure $\mu_\alpha$ which satisfies a coupled 
fixed point equation given in \cite[Theorem 1.4]{bag08}. Later, 
\cite{bcc11} established another characterization of $\mu_\alpha$ in 
terms of the Poisson Weighted Infinite Tree which provided more 
information about properties of $\mu_\alpha$, for example, its absolute 
continuity for $1 < \alpha < 2$ \cite[Theorem 1.6]{bcc11}. The largest 
eigenvalues of these matrices were proven to converge to Poisson point 
processes under suitable scaling conditions \cite{So04}. These results 
were later extend for tail decay in the region $2 < \alpha < 4$ 
\cite{abp09} and a recent generalization of these extreme eigenvalue 
results can be found in \cite{bchj21}. There has also been some 
extraordinary progress in the study of local eigenvalue statistics, 
namely, bulk universality has been proven for $0< \alpha < 2$ 
\cite{ALY18} and for random matrices with $2+\epsilon$ moments 
\cite{A19}, in both of these cases the local eigenvalue statistics fall 
in the GOE universality class. For an overview and a brief survey of 
how heavy-tailed matrices differ from light-tailed ensembles see 
\cite{G18}. Also, while our focus is symmetric matrix ensembles, there 
are several results pertaining to other models such as non-Hermitian 
matrices with heavy-tailed entries, for instance, the recent paper 
\cite{COR20} analyzes a heavy-tailed elliptic random matrix ensemble 
and contains a discussion of the known results for the spectra and 
eigenvector statistics of heavy-tailed random matrix models.

The topic of interest in this paper is the fluctuation at the global 
scale for a subclass of such heavy-tailed real symmetric random 
matrices. Many aspects of the behavior of such fluctuations have been 
understood for all $\alpha$. In all cases, they form a Gaussian process 
with a known expression for the covariance see \cite{bggm14} for 
$\alpha < 2$, \cite{bgm16}) for $2 < \alpha < 4$, and \cite{BS10} for 
the case $\alpha > 4$. Each of these three cases corresponds to a 
genuinely different regime, with both the scaling and the covariance 
changing abruptly at the transition points. For example, when $\alpha 
\in (0,2)$ it is known that the fluctuations of linear statistics of 
its eigenvalues are of order $N^{-\frac{1}{2}}$ rather than the 
$N^{-1}$ behavior of the Wigner case, and for $2 < \alpha < 4$, the 
fluctuations are of order $N^{-\frac{\alpha}{4}}$ --- note that the 
exponent linearly interpolates between the $\alpha$-stable regime and 
the Wigner regime. In the case $0 < \alpha < 2$ the covariance of the 
resolvent is given by coupled fixed point equations that are formidable 
to analyze --- even more so than the spectral measure since they rely 
on an understanding of the fixed point equations for the measure 
$\mu_\alpha$.

We will concentrate on matrices whose entries have tail behavior given 
by $2 < \alpha < 4$. The limiting spectral measure of such matrices is 
still the semicircle law, which is a measure supported on a single 
interval, and therefore is a natural case to compare to the pattern of 
results discussed in the previous section. This class of matrices 
provides a natural starting point in understanding the fluctuations of 
heavy-tailed random matrix eigenvalues.

The main achievement of the present paper is to compute in closed form 
and interpret the covariance of $\mathring{X}_N(f)$ in the context of 
random matrices whose entries have heavier tails with $2 < \alpha < 4$. 
As a starting point of our analysis we adopt the double Laplace 
transform integral formula derived in \cite{bgm16}. We compute the 
integral to arrive at a simpler expression for the covariance of 
$\mathring{X}_N(f)$ where $f(x) = (z-x)^{-1}$ (Theorem 
\ref{t:maintheorem}) which makes the dependence on $\alpha$ and $m$ 
(the Stieltjes transform of the semicircle law) easy to see and 
interpret. In Theorem \ref{thm:genPlemelj} and Corollary 
\ref{t:plemelj} we use our new formula to extract the integral kernel 
associated with this covariance. The resulting integral kernel 
demonstrates the impact of large eigenvalues typical of heavy tailed 
matrix models. 

The rest of the paper is organized as follows, Section~\ref{s:proof1} 
contains the proof of Theorem~\ref{t:maintheorem}, 
Section~\ref{s:proofgenPlemelj} the proof of 
Theorem~\ref{thm:genPlemelj} and Section~\ref{s:corPlem} contains the 
proof of Corollary~\ref{t:plemelj}. The Appendix contains elementary 
integral identities we use in our proofs.

\subsection*{Acknowledgments} We a grateful to the organizers of
IAS/Park City Mathematics Institute 2017 (NSF grant DMS-1441467), where
this collaboration and project originated. A.~M. was supported by the
Royal Society [UF160569]. A.~L. was supported by NSF grant DMS-1646108.

\section{Matrix Model and  Past Results}
In this paper, $\BA_N$ is a sequence of $N\times N$ Hermitian random 
matrix whose entries are i.i.d have first two moments finite, but not 
necessarily any higher moments, in addition to a power law tail decay 
condition.
\begin{definition}
\label{def:matmodel}
Define the sequence of $N\times N$ matrices
\begin{equation*}
\BA_N = [a_{ij}]_{1\leq i,j \leq N} = \left[
\frac{x_{i,j}}{\sqrt{N}}\right]_{1 \leq i,j \leq N},
\end{equation*}
with
\begin{itemize}
\item The $x_{i,j}$, $1 \leq i \leq j$, are i.i.d real random variables 
with mean 0 and variance 1 such that for a certain $\alpha \in (2, 4)$ 
and a certain $c > 0$, as $t \to \infty$,
\begin{equation}
\label{condition::x-ij}
\Prob(|x_{i,j}| > t) \sim \frac{c}{-\Gamma(1 - \frac{\alpha}{2})}t^{-\alpha},
\end{equation}
or
\item $x_{i,j} = x^R_{ij}/\sqrt{2} + ix^{I}_{ij}/\sqrt{2}$ for $1 < i < 
j$ and $x_{ii} = x^R_{ii}$ where $x^I_{ij}$ and $x^R_{ij}$ are i.i.d 
real symmetric random variables with mean 0 and variance 1 that satisfy 
\eqref{condition::x-ij}.
\end{itemize}
\end{definition}

For the above matrix model, the semicircle law for the eigenvalues 
still holds. It was shown in \cite{bgm16} that the spectral statistic
\begin{equation*}
\frac{1}{N^{1 -\frac{\alpha}{4}}}(\Tr \Res(z) - \E \Tr \Res(z)),
\end{equation*}
converges weakly to a centered Gaussian process $X_z$ defined for $z\in
\C\backslash\R$ where
\begin{equation*}
\Res(z) = (z \Id_N - \BA_N)^{-1}.
\end{equation*}
We restate this result, which is the foundation of our calculations in
this paper.
\begin{theorem}[\cite{bgm16}]\label{t:BGM}
For
\begin{equation*}
\Res(z) = (z \Id_N - \BA_N)^{-1},
\end{equation*}
with $\BA_N$ as above, the process
\begin{equation*}
\frac{1}{N^{1 -\frac{\alpha}{4}}}(\Tr \Res(z) - \E \Tr \Res(z)),
\end{equation*}
converges to a complex Gaussian centered process $(X_z)_{z\in \C
  \backslash \R}$ with covariance defined by the fact that $X_{\bar z}
= \overline{X_z}$ and that for any $z, w \in \C \backslash \R$,
$\E[X_z X_{w}] = C(z,w)$, for
\begin{multline}
C(z,w) := \\ \int_{0}^\infty\int_0^\infty \partial_z \partial_{w}\bigg\{[(K(z,t)
  + K(w,s))^{\alpha/2} - (K(z,t)^{\alpha/2} + K(w,s)^{\alpha/2})]
\\ \times \exp(\sgn_z itz - K(z,t) + \sgn_{w} isw - K(w,s))
\bigg\}\frac{c \diff t \diff s}{ 2 t s} \label{eqn::half-heavycovar}
\end{multline}
where $c$ and $\alpha$ are as in \eqref{condition::x-ij}, $\sgn_z =
\sgn(\Im z)$ and $K(z,t) := \sgn_z it m(z)$, $m(z)$ being the
Stieltjes transform of the semicircle law with support $[-2,2]$.
\end{theorem}

The branch cut associated to the fractional power is always the
principal branch cut. This is in contrast to the choice for the
Stieltjes transform, where in the formula:
\begin{equation*}
m(z) = \frac{z - \sqrt{z^2 - 4}}{2},
\end{equation*}
the branch cut is taken to be on the positive real axis. We will remind
the reader of the appropriate branch cut when any branch cut
manipulations are performed.

\section{Main Results}

Our first result is a computation of the integral in
\eqref{eqn::half-heavycovar}.
\begin{theorem}\label{t:maintheorem}
Taking $C(z, w)$ as in Theorem \ref{t:BGM}, we have that
\begin{equation}
C(z,w) = \frac{2 c \pi  m(z)m'(z)m(w)m'(w)}{k_\alpha\sin\big(
\frac{\pi\alpha}{2} \big) (m^2(z) - m^2(w))}
\Big[ \big(-m(z)^2\big)^{\frac{\alpha}{2} - 1} -
\big(-m(w)^2\big)^{\frac{\alpha}{2} - 1}\Big]
\end{equation}
where $c$ and $\alpha$ are as in \eqref{condition::x-ij}, $\sgn_z =
\sgn(\Im z)$, $m(z)$ is the
Stieltjes transform of the semicircle law with support $[-2,2]$, and
\begin{equation}
\label{eqn:normalization}
k_\alpha := 
\frac{\Gamma\big(2-\frac{\alpha}{2}\big)}{\frac{\alpha}{2}\big( 
\frac{\alpha}{2} - 1\big)}.
\end{equation}
\end{theorem}
\begin{remark}
\label{rem:branchcut}
In a further simplification we can write
\begin{multline}
C(z,w) =
-\frac{2 c \pi  m(z)m'(z)m(w)m'(w)}{k_\alpha\sin\big(
\frac{\pi\alpha}{2} \big) (m^2(z) - m^2(w))}\times\\
\bigg( \exp\Big[\frac{(-\sgn z)\pi \alpha i}{2}\Big]m(z)^{\alpha - 2} -
\exp\Big[\frac{(-\sgn w)\pi \alpha i}{2}\Big]m(w)^{\alpha - 2}\bigg)
\end{multline}
\end{remark}
The following Theorem expresses the covariance $C(z,w)$ in an 
alternative integral form.
\begin{theorem}
\label{thm:genPlemelj}
Let $c$ and $k_\alpha$ be defined as in equations~\eqref{condition::x-ij}
and~\eqref{eqn:normalization} respectively.
Let $\psi, \phi \in C_b(\R)$ be bounded continuous functions  and
define
\[
\psi \otimes \phi: \R^2 \to \R \qquad \psi\otimes\phi:= \psi(x)\phi(y) \qquad (x,y)\in \R^2.
\]
Then letting, $z = E+i\eta_1$ and $w = F+i\eta_2$, the pairing
\begin{equation*}
-\frac{1}{4\pi^2}\iint_{\R^2} \bigg\{ C(z, w) + C(\bar{z},\bar{w}) -  
C(\bar{z},w) - C(z,\bar{w}) \bigg\}\psi(E)\phi(F)\,\diff E \,\diff F
\end{equation*}
converges as $\eta_1, \eta_2\downarrow 0$ to the pairing
\begin{equation*}
\langle \Ker_\alpha, \phi\otimes \psi \rangle := \frac{c}{ k_\alpha}\int_0^\infty
\frac{\Lambda_u(\psi)\Lambda_u(\phi)}{ u^{1+\alpha}}\, \diff u
\end{equation*}
where $\Lambda_u$ is a measure indexed by $u$ defined by
\begin{multline*}
\Lambda_u :=  \frac{1}{\pi}\frac{2u^4+2u^2 - u^2x^2}{ \sqrt{4-x^2} \Big(
    -u^2x^2 + u^4 +2u^2+1\Big)}\mathbf{1}_{|x|\leq 2}+
  \big(\delta_{u+u^{-1}}+\delta_{-u-u^{-1}}\big)\mathbf{1}_{u\geq 1}.
\end{multline*}
and $\delta_{x_0}$ is the usual point measure $\delta_{x_0}(\phi)= 
\phi(x_0)$. In particular for $\psi(x) = (z-x)^{-1}$ and $\phi(y) = 
(w-y)^{-1}$ (where we extend $\Ker_\alpha$ to complex-valued bounded 
real continuous functions by linearity) we have
\[
\langle \Ker_\alpha, \phi\otimes \psi\rangle = C(z,w).
\]
\end{theorem}

The above result suggests that for a finite collection of test
functions $\psi_{1}$, $\ldots$, $\psi_m \in \mathcal{F}_\alpha$
where $\mathcal{F}_\alpha$ contains at least the subset
$\mathcal{G}_\alpha$ of finite linear combinations of functions
of the form $f(x) = (z-x)^{-1}$ for some $z \in \C \backslash \R$, the
vector
\[
N^{1-\frac{\alpha}{4}}\Big( \Tr \psi_1(\BA_N) -  \E\Tr \psi_1(\BA_N),
\ldots, \Tr \psi_m(\BA_N) - \E\Tr\psi_m(\BA_N)\Big)
\]
converges weakly to a centered Gaussian vector whose covariance matrix 
has entries $ \langle\Ker_\alpha , \psi_i \otimes \psi_j\rangle$. We 
conjecture the class $\mathcal{F}_\alpha$ is the space of all functions 
$\psi$ for which $\langle \Ker_\alpha, \psi\otimes \psi\rangle$ is 
finite. In order to prove such a Central Limit Theorem, the following 
approach appearing in \cite{Sh11} seems accessible given the explicit 
form of $\Ker_\alpha$:  if we can establish a variance bound
\[
N^{2-\frac{\alpha}{2}} \E\bigg[\Big\{\Tr \psi(\BA_N) - 
\E\Tr\psi(\BA_N)\Big\}^2\bigg] \leq  M \langle \Ker_\alpha, \psi\otimes 
\psi\rangle,
\]
where $M > 0$ is a constant independent of $N$, then by 
\cite[Proposition 3]{Sh11} combined with the known Central Limit 
Theorem~\ref{t:BGM} it follows that the the Central Limit Theorem holds 
for $\psi \in \mathcal{F}_\alpha$ so long as the space 
$\mathcal{G}_\alpha$ is dense in $\mathcal{F}_\alpha$.

More qualitatively, our result helps contextualize the fluctuations of
the eigenvalues of $\BA_N$ at a macroscopic scale with other known
results for eigenvalue statistics for this model. Consider, for
example, the limiting distribution of the largest eigenvalues of
$\BA_N$ computed in the paper \cite[Section 4]{abp09}; the eigenvalues
of $\BA_N$ were shown to converge to a Poisson process under
appropriate scaling
\begin{equation*}
\begin{split}
\sum_{j=1}^N &\delta_{\tilde{b}_N^{-1}\lambda_j(\BA_N)}(\diff x)
\mathbf{1}_{\lambda_j(\BA_N) > 0} \Rightarrow \mathcal{P}_\alpha,\\
\tilde{b}_N &:=\bigg[\frac{N(N+1)}{2}\bigg]^{\frac{1}{\alpha}}
\frac{1}{\sqrt{N}} \sim\frac{N^{\frac{2}{\alpha} -
\frac{1}{2}}}{2^\frac{1}{\alpha}},
\end{split}
\end{equation*}
where $\mathcal{P}_\alpha$ is a Poisson measure on $(0,\infty)$ with 
intensity measure
\begin{equation*}
\rho_\alpha(x) := \frac{\alpha}{x^{1+\alpha}}.
\end{equation*}
Theorem~\ref{thm:genPlemelj} above shows that this intensity measure
appears in the covariance formula $C(z,w)$. In particular, note that
the point masses in the measure $\Lambda_u$ are located at $u+ u^{-1}$
and $-u-u^{-1}$ for $u\geq 1$ taking values on $[-2,2]^c$ which are
outside the support of the semicircle density. If $\psi$ and $\phi$
were test functions supported outside of $[-2,2]$ then the only
contribution in the above covariance would be given by these point
masses integrated to a measure proportional to the intensity measure
$\rho_\alpha$ above. It was observed in \cite{abp09} that the
behavior of the extreme eigenvalues exhibited different behaviors in
the regimes $\alpha > 4$,  $2 < \alpha < 4$ and the genuinely
heavy-tail case of $\alpha < 2$. This relationship parallels the
changing behavior of the fluctuation of linear statistics of the
eigenvalues noted in \cite{bgm16}.  The original covariance
formula presented in Theorem~\ref{thm:genPlemelj} did not provide any
insight into this phenomenon and its relationship to other known
features about the eigenvalue distribution of this class of matrices.
We also provide the following variant of the above formula which may be
easier to use for explicit or numerical evaluation.
\begin{corollary}
\label{t:plemelj}
Let $c$ and $k_\alpha$ be defined as in 
equations~\eqref{condition::x-ij} and~\eqref{eqn:normalization} 
respectively and let $m_{\pm}(E) = \lim_{\eta\downarrow 0} m(E\pm 
i\eta)$. We may write the action of $\langle \Ker_\alpha , \psi\otimes 
\phi\rangle$ as the integral of $\psi(E)\phi(F)$ against the following 
distribution
\[
\Ker_\alpha(E,F) = \frac{c}{k_\alpha}\begin{cases}
\frac{|m_+(E)|^{\alpha-1}(2+m_+(E)^2(2-F^2))m_+'(E)}{\pi(-m_+(E)^{2} 
F^2+(m_+(E)^{2}+1)^2)\sqrt{4-F^2}} &\hbox{if $|E| < 2$ and $|F|>2$}\\ 
\frac{|m_+(F)|^{\alpha-1}(2+m_+(F)^2(2-E^2))m_+'(F)}{\pi(-m_+(F)^{2} 
E^2+(m_+(F)^{2}+1)^2)\sqrt{4-E^2}} &\hbox{if $|E| > 2$ and $|F|<2$}\\ 
(\delta_{E-F} + \delta_{E+F}) |m_+(E)|^{\alpha - 1} |m_+'(E)|& \hbox{if 
$|E|>2$ and $|F| > 2$},
\end{cases}
\]
and for $|E|,|F|< 2$,
\begin{multline*}
\Ker_\alpha(E,F) =
\frac{-c }{\pi k_\alpha \sin\big(\frac{\pi \alpha}{2}\big)}\times\\
\bigg\{
\frac{m_+(E)m'_+(E)m_+(F)m'_+(F)}{ m^2_+(E) - m^2_+(F)}
\bigg[\Big(-m_+(E)^2\Big)^{\frac{\alpha}{2} - 1}
- \big(-m_+(F)^2\big)^{\frac{\alpha}{2} -1 }\bigg]
\\
+\frac{m_-(E)m'_-(E)m_-(F)m'_-(F)}{ m^2_-(E) - m^2_-(F)}
\bigg[\Big(-m_-(E)^2\Big)^{\frac{\alpha}{2} - 1}
- \big( -m_-(F)^2\big)^{\frac{\alpha}{2} - 1}
\bigg]
\\
-\frac{m_+(E)m'_+(E)m_-(F)m'_-(F)}{ m^2_+(E) - m^2_-(F)}
\bigg[\Big(-m_+(E)^2\Big)^{\frac{\alpha}{2} - 1}
- \Big(-m_-(F)^2\Big)^{\frac{\alpha}{2} - 1}
\bigg]
\\
-\frac{m_-(E)m'_-(E)m_+(F)m'_+(F)}{ m^2_-(E) - m^2_+(F)}
\bigg[\Big(-m_-(E)^2\Big)^{\frac{\alpha}{2} - 1}
-  \Big(-m_+(F)^2\Big)^{\frac{\alpha}{2} - 1}\bigg]
\bigg\},
\end{multline*}
In particular, for $\alpha = 3$ we have
\[
\Ker_3 (E, F)=
\frac{15c}{8\sqrt{\pi}}
\begin{cases}
\frac{1}{\pi}
\bigg(1 + \frac{(F^2 - 2)(E^2-2)}{\sqrt{4 -
F^2}\sqrt{4-E^2}}\bigg) \frac{1}{\sqrt{4 - E^2} + \sqrt{4 - F^2}},
&\hbox{if $|E|<2$ and $|F| < 2$},\\
\frac{|m_+(E)|^{2}(2+m_+(E)^2(2-F^2))m_+'(E)}{\pi(-m_+(E)^{2} F^2+(m_+(E)^{2}+1)^2)\sqrt{4-F^2}}
&\hbox{if $|E| < 2$ and $|F|>2$},\\
\frac{|m_+(F)|^{2}(2+m_+(F)^2(2-E^2))m_+'(F)}{\pi(-m_+(F)^{2} E^2+(m_+(F)^{2}+1)^2)\sqrt{4-E^2}}
&\hbox{if $|E| > 2$ and $|F|<2$},\\
(\delta_{E-F} + \delta_{E+F})|m(E)|^2m'(E), &\hbox{if $|E|>2$ and $|F|>2$}.
\end{cases}
\]
\end{corollary}

\begin{figure}[t]
\centering{\includegraphics[scale=0.6]{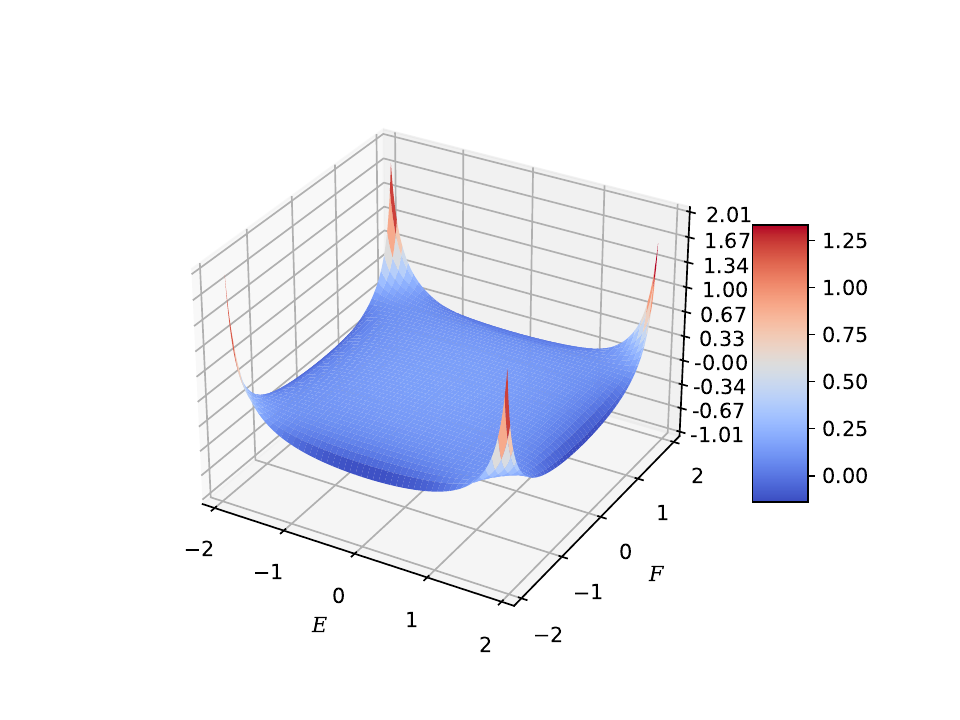}
\caption{Plot of $\mathcal{K}_3$  with $|E|, |F| < 2$ and $c = \frac{8\sqrt{\pi}}{15}$.}
}
\end{figure}

\section{Proof of Theorem \ref{t:maintheorem}: Rewriting the Integral}
\label{s:proof1}

\begin{proof}
By the fixed point equation representation of $m(z)$:
\begin{equation*}
m(z) = \frac{1}{z - m(z)},
\end{equation*}
we can write
\begin{equation}
\label{eqn:covar-rewrite}
C(z,w) = \frac{c}{2} \int_0^\infty \int_0^\infty \partial_z
\partial_w\cL(z,w;t,s)\,\diff t \, \diff s,
\end{equation}
where
\begin{multline*}
\cL(z,w; t,s) := \frac{1}{ts}\bigg\{ \big(it \sgn_z m(z) + is \sgn_w
m(w)\big)^\frac{\alpha}{2} - \big(i t \sgn_z
m(z)\big)^\frac{\alpha}{2} \\ - \big( i s \sgn_w
m(w)\big)^\frac{\alpha}{2}\bigg\} \exp\bigg(\frac{it \sgn_z}{m(z)} +
\frac{is \sgn_w}{m(w)}\bigg).
\end{multline*}
we have that
\begin{equation*}
\Re\bigg(\frac{i\sgn_z}{m(z)}\bigg) < 0 \qquad \text{and}\qquad
\Re\bigg(\frac{i\sgn_w}{m(w)}\bigg) < 0,
\end{equation*}
so for each $z, w \in \C \backslash \R$, $\cL(z,w; t,s)$ decays in $t$
and $s$ exponentially. We wish to pull out the derivatives of $z$ and
$w$ out of the integral in \eqref{eqn:covar-rewrite}. To do so,
observe that by the Cauchy integral formula,
\begin{equation*}
\partial_z \partial_w \cL(z,w;t,s) = - \frac{1}{4\pi^2} \oint_{\cC_z}
\oint_{\cC_w} \frac{\cL(\rho,\gamma; t,s)}{(\rho - z)^2(\gamma - w)^2}
\,\diff \rho \, \diff \gamma,
\end{equation*}
where $\cC_z$ and $\cC_w$ are contours containing $z$ and $w$
respectively --- we take them to be balls of a small radius so they
avoid the branch cut. Now we have
\begin{multline*}
\int_0^\infty \int_0^\infty \partial_z \partial_w \cL(z,w;t,s)\,\diff
t \, \diff s = \\ -\int_0^\infty \int_0^\infty \frac{1}{4\pi^2}
\oint_{\cC_z} \oint_{\cC_w} \frac{\cL(\rho,\gamma; t,s)}{(\rho -
  z)^2(\gamma - w)^2} \,\diff \rho \, \diff \gamma\,\diff t\,\diff s,
\end{multline*}
now by choosing contours to be balls of radius $\delta$ we have
\begin{multline*}
\int_0^\infty \int_0^\infty \oint_{\cC_z} \oint_{\cC_w}
\bigg|\frac{\cL(\rho,\gamma; t,s)}{(\rho - z)^2(\gamma - w)^2}\bigg|
\,\diff \rho \, \diff \gamma = \\ \frac{1}{\delta^2}\int_0^\infty
\int_0^\infty \oint_{\cC_z} \oint_{\cC_w} |\cL(\rho,\gamma; t,s)|
\,\diff \rho \, \diff \gamma\, \diff t \,\diff s,
\end{multline*}
we claim that the integral in $t$ and $s$ of $|\cL(\rho,\gamma; t,s)|$
is uniformly bounded for all $\rho$ and $\gamma$ in the chosen
contours.  This allows us to rearrange the original integrals using
Fubini's Theorem,
\begin{equation*}
C(z,w) = \frac{c}{2} \partial_z \partial_w \int_0^\infty \int_0^\infty
\cL(z,w;t,s)\,\diff t \, \diff s.
\end{equation*}

Next, observe that
\begin{equation*}
\Re(- it \sgn_z m(z)) <0 ,\quad\hbox{and}\quad \Re(-i s \sgn_w m(w)) <
0,
\end{equation*}
so we may use Lemma \ref{lem:pvintegral1} to write
\begin{multline*}
\big(it \sgn_z m(z) + is \sgn_w m(w)\big)^\frac{\alpha}{2} -\big(i t
\sgn_z m(z)\big)^\frac{\alpha}{2} - \big( i s \sgn_w
m(w)\big)^\frac{\alpha}{2}= \\ \int_0^\infty\frac{1}{k_\alpha r^{1 +
\frac{\alpha}{2}}} \Big\{\exp\big(-ir\{t\sgn_z m(z)+s\sgn_w
m(w)\}\big)\\ -\exp(-ir t\sgn_z m(z))-\exp(-ir s\sgn_w m(w))+1\Big\}\,
\diff r,
\end{multline*}
obtaining the equality
\begin{multline*}
\cL(z,w;t,s) = \int_0^\infty \frac{1}{k_\alpha r^{1 +
\frac{\alpha}{2}}ts} \bigg\{\exp\bigg[it\sgn_z\bigg(-r
m(z)+\frac{1}{m(z)}\bigg)\bigg] -\exp\bigg[ \frac{i t
\sgn_z}{m(z)}\bigg]\bigg\}\\ \bigg\{\exp\bigg[is \sgn_w\bigg(-r m(w) +
\frac{1}{m(w)}\bigg)\bigg] - \exp\bigg[ \frac{is
\sgn_w}{m(w)}\bigg]\bigg\}\,\diff r.
\end{multline*}
Using the above representation, we can integrate $t$ and $s$ first by
applying Fubini's Theorem again. Observe that
\begin{multline*}
\frac{\exp\big[it\sgn_z\big(-r m(z)+\frac{1}{m(z)}\big)\big] -\exp\big[
\frac{i t \sgn_z}{m(z)}\big]}{t}\\ = -ir\sgn_z
m(z)\int_0^1\exp\bigg[it\sgn_z\bigg(-r\nu m(z)
+\frac{1}{m(z)}\bigg)\bigg] \,\diff \nu,
\end{multline*}
so
\begin{multline*}
\int_0^\infty \frac{\exp\big[it\sgn_z\big(-r
m(z)+\frac{1}{m(z)}\big)\big] -\exp\big[ \frac{i t
\sgn_z}{m(z)}\big]}{t} \,\diff t \\ = -ir\sgn_z m(z)\int_0^1
\int_0^\infty \exp\bigg[it\sgn_z\bigg(-r\nu m(z)
+\frac{1}{m(z)}\bigg)\bigg]\,\diff t \,\diff \nu, \\ =
r m(z) \int_0^1
\frac{\diff \nu}{\big(-r\nu m(z)+\frac{1}{m(z)}\big)} \\ =
- \bigg\{\log\bigg(-r m(z) + \frac{1}{m(z)}\bigg) -
\log\bigg(\frac{1}{m(z)}\bigg)\bigg\},
\end{multline*}
where we have used that $-r\nu m(z) + \frac{1}{m(z)}$ does not cross the
principal branch cut of the logarithm. We conclude
\begin{multline*}
\int_0^\infty \int_0^\infty \cL(z,w;t,s) \,\diff t \, \diff s = \\
\int_0^\infty \frac{1}{k_\alpha r^{1+\frac{\alpha}{2}}} \bigg(
\log\bigg[- r m(z) + \frac{1}{m(z)}\bigg] -
\log\bigg[\frac{1}{m(z)}\bigg]\bigg) \\ \times\bigg(\log\bigg[ -r m(w)
+ \frac{1}{m(w)}\bigg] - \log\bigg[\frac{1}{m(w)}\bigg]\bigg)\,\diff
r.
\end{multline*}
Insert the partial derivatives with respect to $z$ and $w$ back
into the integrand above to obtain
\begin{multline}
	\label{e:intrep}
C(z,w) = \frac{c}{2}\int_0^\infty
\frac{4r^{1-\frac{\alpha}{2}}}{k_\alpha}\frac{ m(z)m'(z) m(w)
m'(w)}{(r m^2(z) -1) (rm^2(w) - 1)} \,\diff r,\\
= \frac{c}{2}\int_0^\infty \frac{4 r^{\frac{\alpha}{2} - 1}}{k_\alpha}
\frac{m(z)m'(z)m(w)m'(w)}{(r-m^2(z))(r-m^2(w))}\,\diff r.
\end{multline}
We apply Lemma \ref{lem:resIntegral} to obtain
\begin{equation}
\label{eqn:covarFinal}
C(z,w) =
\frac{2c\pi m(z)m'(z) m(w)m'(w)}{k_\alpha\sin\big(\frac{\pi
\alpha}{2}\big)}\bigg[\frac{\big(-m^2(z)\big)^{\frac{\alpha}{2} - 1} -
\big(-m^2(w)\big)^{\frac{\alpha}{2} -1 }}{m^2(z) - m^2(w)}\bigg].
\end{equation}
Simplifying this further, we see that if $\Im z > 0$ then $\arg [m(z)] 
\in (-\pi, 0)$ yielding that $\arg[-m^2(z)] = \arg[\exp({i \pi})m^2(z)] 
\in (-\pi, \pi)$ which implies
\[
\big(-m^2(z)\big)^{\frac{\alpha}{2}-1} = \exp\bigg[i \pi\Big(1 - 
\frac{\alpha}{2}\Big)\bigg]m(z)^{\alpha-2}.
\]
On the other hand, if $\Im z < 0$ then $\arg m(z) \in (0, \pi)$ 
yielding that $\arg[-m^2(z)] = \arg[\exp(-i \pi)m^2(z)] \in (-\pi, 
\pi)$ which implies
\[
(-m^2(z))^{\frac{\alpha}{2}-1} = \exp\bigg[-i \pi\Big(1 - 
\frac{\alpha}{2}\Big)\bigg] m(z)^{\alpha-2}.
\]
Applying these observations in equation~\eqref{eqn:covarFinal} proves 
Remark~\ref{rem:branchcut}. Observe that $C(z,w)$ is analytic in both 
$z$ and $w$ so long as $z$ and $w$ are not on the real line.
\end{proof}

\section{Proof of Theorem~\ref{thm:genPlemelj}}
\label{s:proofgenPlemelj}

Recall from Section~\ref{s:proof1} the formula
\begin{equation*}
\begin{split}
C(z,w) &= \frac{c}{2}\int_0^\infty \frac{4r^{\frac{\alpha}{2}-1}}{k_\alpha}
\frac{m(z)m'(z)m(w)m'(w)}{\big(r-m^2(z)\big)\big(r-m^2(w)\big)}\,\diff
r\\ &= \frac{c}{k_\alpha} \int_0^\infty \frac{1}{u^{1 + \alpha}}
\frac{4u^4 m(z)m'(z)m(w)m'(w)}{\big(1-u^2m^2(z)\big)\big(1 - u^2
  m^2(w)\big)} \,\diff u,
\end{split}
\end{equation*}
where we have mapped $r$ to $u^{-2}$ so that the intensity measure of
limiting point process of the largest eigenvalues of $\BA_N$ appears.
From this representation
\begin{align*}
C(z,w) - C(\bar{z},w) &= \frac{2ic}{k_\alpha}
\int_0^\infty\frac{1}{u^{1+\alpha}} \Im
\bigg\{\frac{2u^2m(z)m'(z)}{1-u^2m^2(z)}\bigg\}
\frac{2u^2m(w)m'(w)}{1-u^2m^2(w)} \,\diff u,\\ C(z,\bar{w}) -
C(\bar{z},\bar{w}) &= \frac{2ic}{k_\alpha}
\int_0^\infty\frac{1}{u^{1+\alpha}} \Im
\bigg\{\frac{2u^2m(z)m'(z)}{1-u^2m^2(z)}\bigg\}
\overline{\frac{2u^2m(w)m'(w)}{1-u^2 m^2(w)}}\,\diff u,
\end{align*}
hence
\begin{multline*}
-\frac{1}{4\pi^2} \bigg\{ C(z, w) + C(\bar{z},\bar{w}) - C(\bar{z},w)
- C(z,\bar{w}) \bigg\}\\ = \frac{c}{\pi^2 k_\alpha} \int_0^\infty
\frac{1}{u^{1+\alpha}} \Im
\bigg\{\frac{2u^2m(z)m'(z)}{1-u^2m^2(z)}\bigg\}
\Im\bigg\{\frac{2u^2m(w)m'(w)}{1-u^2m^2(w)}\bigg\}\,\diff u.
\end{multline*}
This representation will be the basis for our proof of
Theorem~\ref{thm:genPlemelj}.  We need the following Lemmas in order
to proceed.

\begin{lemma}
\label{lem:argvar}
Let $\psi \in C_b(\R)$ and let $\zeta = x + i\eta$ with $\eta >
0$. For fixed $\eta$, define the functions
\begin{align*}
r_\eta(\theta) &:= \frac{\eta - \sqrt{\eta^2 + 4
    \sin^2(\theta)}}{2\sin\theta} \qquad \theta \in
(-\pi,0)\\ E_\eta(\theta) &:= \bigg(r_\eta(\theta) +
\frac{1}{r_\eta(\theta)}\bigg)\cos\theta \qquad \theta \in (-\pi,0)
\end{align*}
then for every $\eta>0$,
\begin{multline*}
\int_\R\frac{1}{\pi}\Im \bigg\{\frac{2u^2m(x + i \eta)m'( x+
  i\eta)}{1-u^2m^2(x+ i\eta)}\bigg\} \psi(x)\,\diff x \\ =
-\frac{1}{\pi}\int_{-\pi}^0 \psi\big(E_{\eta}(\theta)\big)
\frac{\diff}{\diff
  \theta}\arg\big\{1-u^2r_{\eta}^2(\theta)\exp(i2\theta)\big\}\,\diff
\theta.
\end{multline*}
\end{lemma}
\begin{proof}
For all $x\in \R$, $-m^2(\zeta)$ is in the complex disk with negative
reals removed, $\mathbb{D}\backslash{[-1,0]}$, therefore the function
$1-u^2m^2(\zeta)$ never crosses the principal branch cut of the
logarithm and
\begin{equation*}
\frac{\diff}{\diff \zeta } \log\big\{1-u^2m^2(\zeta)\big\} =
-\frac{2u^2m(\zeta)m'(\zeta)}{1-u^2m^2(\zeta)}.
\end{equation*}
By the Cauchy-Riemann equations,
\begin{equation}
\label{eqn:cauchy-riemann}
\Im\bigg[\frac{\diff}{\diff \zeta} \log\big\{1-u^2m^2(\zeta)\big\}
  \bigg] = \frac{\partial}{\partial x} \Im
\log\big\{1-u^2m^2(\zeta)\big\},
\end{equation}
note that the imaginary part of the logarithm is the principal branch
of the argument
\[
\Im \log \big\{1 - u^2 m^2(x+i\eta)\big\} = \arg\big\{1-u^2
m^2(x+i\eta)\big\}.
\]
Let $\psi\in C_b(\R)$, by equation~\eqref{eqn:cauchy-riemann}
\begin{multline}
\label{eqn:argumentintegral}
\int_\R\frac{1}{\pi}\Im \bigg\{\frac{2u^2m(x + i \eta)m'( x+
  i\eta)}{1-u^2m^2(x+ i\eta)}\bigg\} \psi(x)\,\diff x \\ =
-\frac{1}{\pi}\int_\R \psi(x)\frac{\partial}{\partial
  x}\arg\big\{1-u^2m^2(x + i\eta)\big\} \,\diff x.
\end{multline}
Let the polar form of $m(\zeta) = R \exp(i\theta)$ note that $\theta =
\arg\{m(\zeta)\}\in (-\pi,0)$. The rewritten fixed point equation
\[
\zeta = m(\zeta) + \frac{1}{m(\zeta)}
\]
implies
\[
x + i\eta = m(\zeta) + \frac{\overline{m(\zeta)}}{|m(\zeta)|^2} =
\Re\{m(\zeta)\} \bigg( 1 + \frac{1}{|m(\zeta)|^2}\bigg) +
i\Im\{m(\zeta)\} \bigg(1 - \frac{1}{|m(\zeta)|^2}\bigg),
\]
which is equivalent to
\begin{equation}
\label{eqn:param}
x = \bigg(R + \frac{1}{R}\bigg)\cos\theta \quad\hbox{and}\quad \eta =
\bigg( R - \frac{1}{R}\bigg)\sin \theta .
\end{equation}
We use equation \eqref{eqn:param} to express the radius $r$ solely as
a function of $\theta$ and $\eta$
\begin{equation}
\label{eqn:radialquad}
R^2 \sin\theta - R\eta - \sin\theta = 0 \implies R = r_{\eta}(\theta)=
\frac{\eta - \sqrt{\eta^2 +4\sin^2\theta}}{2\sin\theta}.
\end{equation}
then using the function $r_{\eta}(\theta)$ we write $x$ as a function
of $\theta$ and $\eta$
\[
x = E_{\eta}(\theta) = \bigg(r_{\eta}(\theta) +
\frac{1}{r_{\eta}(\theta)} \bigg) \cos\theta.
\]
The function $E_{\eta}(\theta)$ is a bijection from $(-\pi,0)$ to $\R$ 
and is differentiable. We may therefore change variables in 
equation~\eqref{eqn:argumentintegral}
\begin{multline*}
-\frac{1}{\pi}\int_\R \psi(x)\frac{\partial}{\partial
x}\arg\big\{1-u^2m^2(x + i\eta)\big\} \,\diff x\\
=-\frac{1}{\pi}\int_{-\pi}^0 \psi\big(E_{\eta}(\theta)\big)
\frac{\diff}{\diff
\theta}\arg\big\{1-u^2r_{\eta}^2(\theta)\exp(i2\theta)\big\}\,\diff
\theta,
\end{multline*}
which is the required result.
\end{proof}

We wish to bound the expression in Lemma~\ref{lem:argvar} for
arbitrary test functions $\psi$ and arbitrary $u\geq 0$ and $\eta>0$.

\begin{lemma}
\label{lem:propofcurve}
Let $\eta>0$ and let $r_\eta(\theta)$ be as defined in
Lemma~\ref{lem:argvar}.  The interval $(-\pi,0)$ can be partitioned
into at most 9 intervals for which the function
\begin{equation}
\label{eqn:argumentfunc}
\arg\big\{1 - u^2r_{\eta}^2(\theta)\exp(i2\theta)\big\},
\end{equation}
is piecewise monotonic. When $u\leq 1$ in particular, there exists a
point $\tilde{\theta}_{\eta,u}\in(-\frac{\pi}{2},0)$ such that the
function~\eqref{eqn:argumentfunc} is decreasing on
$(-\pi,-\tilde{\theta}_{\eta,u}-\pi)$, increasing on
$(-\tilde{\theta}_{\eta,u}-\pi,\tilde{\theta}_{\eta,u})$ and
decreasing once more on $(\tilde{\theta}_{\eta,u},0)$. Further the
point $\tilde{\theta}_{\eta,u}$ is defined to be the unique solution
to
\[
-u^2 r_\eta^8(\theta)- (1-u^2)r_\eta^6(\theta) +(5 + u^2)
r_\eta^4(\theta) - (u^2 + 4\eta^2 + 7)r_\eta^2(\theta) + 3=0
\]
over $\theta\in(-\frac{\pi}{2},0)$.
\end{lemma}
\begin{proof}
For $\theta \in (-\frac{\pi}{2},0)$ we have $\sin(-\pi - \theta) =
\sin(\theta)$ further $\exp(i2\big(-\pi - \theta\big)) =
-\exp(-i2\theta)$ from which we conclude
\[
\arg\big\{1 - u^2
r_{\eta}^2(-\pi-\theta)\exp\big(i2(-\pi-\theta)\big)\big\} =
-\arg\big\{1 - u^2r_\eta^2(\theta)\exp(i2\theta)\big\},
\]
therefore it suffices to prove the existence of a $\theta^*_{\eta,u}$
such that the function~\eqref{eqn:argumentfunc} is increasing on
$(-\frac{\pi}{2},\theta^*_{\eta,u})$ and decreasing on
$(\theta^*_{\eta,u},0)$. For what proceeds assume $\theta \in
(-\frac{\pi}{2},0)$.  Let $\arctan$ be the principal branch of the
inverse tangent taking values in $(-\frac{\pi}{2},\frac{\pi}{2}\big]$.
  Observe that
\begin{multline*}
\arg\big\{1 - u^2 r_{\eta}^2(\theta) \exp(i2\theta)\big\} \\ =2
\arctan\Bigg(\frac{ - u^2 r_{\eta}^2(\theta) \sin(2\theta)}{1-u^2
  r_{\eta}^2\cos(2\theta) +
  |1-r^2_{\eta}(\theta)\exp(i2\theta)|}\Bigg).
\end{multline*}
We define
\[
\begin{split}
g(\theta) :=& \frac{1- u^2r_{\eta}^2(\theta) \cos(2\theta)}{ -u^2
  r_{\eta}^2(\theta)\sin(2\theta)},
=\frac{1-u^2r_{\eta}^2(\theta)\big\{1- 2 \sin^2(\theta)\big\}}{-2u^2
  r_\eta^2(\theta)\sin(\theta)\sqrt{1-\sin^2(\theta)}},\\ =& \frac{1 -
  u^2r_\eta^2(\theta)\Big\{1 - 2\frac{\eta^2 r_\eta^2(\theta)}{(1-
    r_\eta^2(\theta))^2}\Big\}} {\frac{2 u^2\eta
    r_\eta^3(\theta)}{1-r_\eta^2(\theta)}\sqrt{1 - \frac{\eta^2
      r_\eta^2(\theta)}{(1-r_\eta^2(\theta))^2}}},\\ =&
\frac{\big(1-r_\eta^2(\theta)\big)^2 -
  u^2r_\eta^2(\theta)\Big\{\big(1-r_\eta^2(\theta)\big)^2- 2\eta^2
  r_\eta^2(\theta)\Big\}} {2\eta u^2
  r_\eta^3(\theta)\sqrt{\big(1-r_\eta^2(\theta)\big)^2 - \eta^2
    r_\eta^2(\theta)}},\\ =& \frac{- u^2r_\eta^6(\theta) +\big\{1 +
  2u^2(1+\eta^2)\big\}r_\eta^4(\theta) - (2+u^2)r_\eta^2(\theta)+1}
     {2\eta u^2
       r_\eta^3(\theta)\sqrt{r_\eta^4(\theta)-(2+\eta^2)r_\eta^2(\theta)
         + 1}},
\end{split}
\]
where we have used the equation~\eqref{eqn:radialquad} to express
$g(\theta)$ as a function of $r_\eta(\theta)$ only.  Observe that
\begin{multline*}
\frac{\diff}{\diff \theta} \arg\big\{1 - u^2 r_{\eta}^2(\theta)
\exp(i2\theta)\big\} \\ = -g'(\theta)\times \frac{2}{\big(g(\theta) +
  \sqrt{1+g^2(\theta)}\big)^2} \bigg(1 +\frac{g(\theta)}{\sqrt{1 +
    g^2(\theta)}}\bigg) \frac{1}{1+
  \frac{1}{\big(g(\theta)+\sqrt{1+g^2(\theta)}\big)^2}},
\end{multline*}
therefore $\frac{\diff}{\diff\theta}\arg\big\{1 -
u^2r_{\eta}^2(\theta) \exp(i2\theta)\big\}$ has the same sign as
$-g'(\theta)$.  Next,
\begin{multline}
-g'(\theta) = -r_\eta'(\theta)\times\\ \bigg\{\frac{-u^2r^{10}_\eta -
  (1-2u^2)r_\eta^8 +6 r_\eta^6-2(u^2+2\eta^2+6)r_\eta^4+(u^2 +
  4\eta^2+10)r_\eta^2 - 3}{ 2\eta u^2 r_\eta^4(r_\eta^4 -
  (\eta^2+2)r_\eta^2 + 1)^{\frac{3}{2}}}\bigg\},
\end{multline}
Note that $-r_\eta'(\theta) > 0$ for all $\theta \in
(-\frac{\pi}{2},0)$, so the sign of $-g'(\theta)$ is the same as the
sign of the numerator of the above expression.  Therefore, it suffices
to study the behavior of the polynomial in the numerator in the
variables $v= r_\eta^2$
\[
p(v) :=-u^2v^5 - (1-2u^2)v^4 +6 v^3-2(u^2+2\eta^2+6)v^2+(u^2 +
4\eta^2+10)v - 3
\]
Note that $p(1) = 0$ so we may write
\begin{align*}
p(v) &= (v-1)q(v),\\ q(v)&:= -u^2 v^4 - (1-u^2)v^3 +(5 + u^2) v^2 -
(u^2 + 4\eta^2 + 7)v + 3.
\end{align*}
For general $u$, the polynomial $q(v)$ admits at most 4 real roots in
the interval $(0,1)$, in between these roots $q$ is either positive or
negative (with the opposite sign for $p(v)$) which proves the first
part of the Theorem.  We now study the roots of $q(v)$ for $u^2 \leq
1$.  Observe that
\[
q'(v) = -4 u^2 v^3 - 3(1-u^2)v^2 + 2(5+u^2)v - u^2 - 4 \eta^2 - 7,
\]
further note that
\[
q'(v+1) = -4u^2 v^3 - 3(3u^2 + 1) v^2 - 4(u^2 - 1) v - 4 \eta^2,
\]
by Budan's Theorem \cite[Theorem 3]{Ak82}, if $s_{q'(v)}$ is the number of sign changes in
the non-zero coefficients of $q'(v)$ and $s_{q'(v+1)}$ is the number
of sign changes in the non-zero coefficients of $q'(v+1)$ then
\[
s_{q'(v)} - s_{q'(v+1)} - \#\{v\in(0,1] : q'(v) = 0\}
\]
is a non-negative even integer. When $u^2 \leq 1$,
$s_{q'(v)} = s_{q'(v+1)} = 2$ so that $q'(v)$ has no roots in $(0,1]$.
Because $q'(0) = -u^2 - 4 \eta^2 - 7 < 0$, it follows that $q'(v) < 0$
for all $v \in(0,1]$. Further, $q(0) = 3>0$ and
$q(1) = -4\eta^2 < 0 $, which implies $q(v)$ is monotonically
decreasing and has exactly one root in $[0,1]$. Since the polynomial
is in the variable $r_\eta^2$, we must verify this root is attained in
the range of $r_\eta^2(\theta)$; to this end note
\[
    v^*=\sup_{\theta\in(-\frac{\pi}{2},0)}r_\eta^2(\theta)
    =\bigg(\sqrt{1+\frac{\eta^2}{4}} - \frac{\eta}{2}\bigg)^2 \implies (1-v^*)^2 = \eta^2 v^*,
\]
and
\[
\begin{split}
  q(1) - q(v^*) &= \int_0^{1-v^*} q'(1-t)\,\diff t, \\
  &= \int_0^{1-v^*}t\big\{4u^2 t^2 - 3(3u^2+1)t+4(1-u^2)\big\}\,\diff t - 4 \eta^2(1-v^*),\\
  & = u^2\eta^4(v^*)^2 - (3u^2+1)(1-v^*)\eta^2 v^* +2(1-u^2)\eta^2 v^* - 4 \eta^2(1-v^*),\\
  & = \big(\eta^4
    u^2+\eta^2(3u^2+1)\big)(v^*)^2+5\eta^2(1-u^2)v^*-4\eta^2 >
  q(1),
\end{split}
\]
so $q(v^*) < 0$ meaning the unique root of $q$  in $[0,1]$ is acheived for $v$ of the
form $v=r_\eta^2(\theta)$ and $\theta\in(-\frac{\pi}{2},0)$.
The second statement of the Theorem now holds because if
$v=\tilde{r}_\eta^2$ is the root of $q(v)$, the sign of $-g'(\theta)$
is positive for $\tilde{r}_\eta^2 < r_\eta^2 < 1$, and negative for
$r_\eta^2 < \tilde{r}_\eta^2$, defining $\tilde{\theta}_{\eta,u}$ so
that $\tilde{r}_\eta = r_\eta(\tilde{\theta}_{\eta,u})$, gives the
required result.
\end{proof}
Using the above Lemma, we are able to obtain the following bound.
\begin{lemma}
\label{lem:integrationbound}
Let $\psi \in  C_b(\R)$ be a bounded continuous test function,
then the following estimate holds
\begin{multline}
\label{eqn:argintestimate}
\frac{1}{\pi}\int_\R\Bigg|\Im\bigg\{\frac{2u^2m(x + i \eta)m'( x+
  i\eta)}{1-u^2m^2(x + i\eta)}\bigg\} \psi(x)\Bigg| \,\diff x
\\ \leq \begin{cases} 18\|\psi\|_{\infty} & \hbox{for all
    $u$},\\ \frac{4\|\psi\|_\infty\arg\{1 - u^2
    r_\eta^2(\tilde{\theta}_{\eta,u})\exp(i2\tilde{\theta}_{\eta,u})\}}{\pi}
  & \hbox{for $u \leq 1$},
\end{cases}
\end{multline}
where $\tilde{\theta}_{\eta,u}$ is defined in
Lemma~\ref{lem:propofcurve}.
\end{lemma}

\begin{proof}
Using Lemma~\ref{lem:argvar}, the absolute value of the above integral
is the same as
\begin{equation}
\label{eqn:absoluteargint}
\frac{1}{\pi}\int_{-\pi}^0
\Bigg|\psi\big(E_\eta(\theta)\big)\frac{\diff}{\diff\theta}\arg\big\{1
- u^2r_\eta^2(\theta)\exp(i2\theta)\big\}\Bigg| \,\diff \theta ,
\end{equation}
now using Lemma~\ref{lem:propofcurve}, we partition $(-\pi,0)$ into
subintervals $P_1$, $\ldots$, $P_k$ with $k \leq 9$, such that
$\arg\{1 - u^2r^2_\eta(\theta)\exp(i2\theta)\}$ is monotonically
increasing or decreasing on each $P_i$. Using linearity and triangle
inequality equation~\eqref{eqn:absoluteargint} is bounded by
\[
\|\psi\|_\infty \sum_{i=1}^k \frac{1}{\pi}\int_{P_i}\bigg|
\frac{\diff}{\diff\theta}\arg\big\{1 -
u^2r_\eta^2(\theta)\exp(i2\theta)\big\}\bigg| \,\diff \theta \leq
2k\|\psi\|_\infty \leq 18 \|\psi\|_\infty .
\]
For $u \leq 1$, we use the same style of argument, except now
equation~\eqref{eqn:absoluteargint} is bounded by
\begin{multline*}
\frac{\|\psi\|_\infty}{\pi} \Bigg(
\int_{-\pi}^{-\tilde{\theta}_{\eta,u} - \pi} - \frac{\diff}{\diff
  \theta}\arg\big\{1 -u^2r_\eta^2(\theta)\exp(i2\theta)\big\}\diff
\theta \\ +\int_{-\tilde{\theta}_{\eta,u} -
  \pi}^{\tilde{\theta}_{\eta,u}} \frac{\diff}{\diff\theta}\arg\big\{1
- u^2 r^2_\eta(\theta)\exp(i2\theta)\big\} \diff\theta \\ +
\int_{\tilde{\theta}_{\eta,u}}^0 -\frac{\diff}{\diff\theta}\arg\big\{1
- u^2 r^2_\eta(\theta)\exp(i2\theta)\big\} \diff\theta \Bigg),
\end{multline*}
which, by $r_\eta(-\pi-\theta) = r_\eta(\theta)$ for all $\theta \in (-\frac{\pi}{2},0)$, equals
\[
\frac{4\|\psi\|_\infty}{\pi} \arg\big\{ 1 -
u^2r_\eta^2(\tilde{\theta}_{\eta,u})\exp(i2\tilde{\theta}_{\eta,u})\big\},
\]
as required.
\end{proof}
With these estimates in place, we may now argue that the limit as
$\eta_1$ and $\eta_2\downarrow 0$ can be taken under the integral with
respect to $u$.

\begin{lemma}
  \label{lem:domconvinu}
Let $z=E+i\eta_1$ and $w= F + i \eta_2$.
For any $\psi, \phi \in C_b(\R)$ we have the following inequalities
\begin{multline*}
\sup_{\eta_1,\eta_2\leq 1}\iint_{\R^2}\frac{|\psi(E)||\phi(F)|}{\pi^2}\Bigg|\Im\Bigg\{\frac{2u^2
  m(z)m'(z)}{1-u^2m^2(z)}\Bigg\}
\Im\Bigg\{\frac{2u^2m(w)m'(w)}{1-u^2m^2(w)}\Bigg\}\Bigg|\,\diff E\, \diff F \\\leq 18^2\|\psi\|_\infty
\|\phi\|_\infty\big(u^4\mathbf{1}_{u<\frac{1}{2}} +
\mathbf{1}_{|u|\geq\frac{1}{2}}\big)
\end{multline*} 
and
\begin{multline*}
\int_0^\infty\frac{1}{u^{1+\alpha}}\sup_{\eta_1,\eta_2\leq
  1}\Bigg\{\iint_{\R^2}\Bigg( \frac{|\psi(E)||\phi(F)|}{\pi^2} \times
\\\Bigg|\Im\Bigg\{\frac{2u^2 m(z)m'(z)}{1-u^2m^2(z)}\Bigg\}
\Im\Bigg\{\frac{2u^2m(w)m'(w)}{1-u^2m^2(w)}\Bigg\}\Bigg|\Bigg)\,\diff
E\, \diff F\Bigg\} \,\diff u \\
\leq \frac{18^2 (2^\alpha)}{4-\alpha}\|\psi\|_\infty \|\phi\|_\infty.
\end{multline*}
\end{lemma}
\begin{proof}
Using the notation of Lemma~\ref{lem:propofcurve}
\[
\sup_{\theta\in(-\pi,0)}\sup_{\eta_j \leq 1}|r_{\eta_j}^2(\theta)| \leq 1,
\]
therefore $1 - u^2 
r_{\eta_j}^2(\tilde{\theta}_{\eta_j,u})\exp(i2\tilde{\theta}_{\eta_j,u})$ 
for any $u < 1$ is away from the branch cut of the argument function 
and its distance to $1$ is bounded by $u^2$. Due to the bound
\[
|\zeta| < \frac{1}{2} \implies|\arg\{1- \zeta\}| \leq |\zeta|,
\]
we have for $|u|<\frac{1}{2}$ 
\[
\arg\big\{1 - u^2r_{\eta_j}^2 (\tilde{\theta}_{\eta_j,u})
\exp(i2\tilde{\theta}_{\eta_j,u})\big\} \leq u^2.
\]
Next, applying
Lemma~\ref{lem:integrationbound} into the integrand we obtain the
upper bound
\begin{multline*}
\iint_{\R^2} |\psi(E)||\phi(F)|\Bigg|\Im\Bigg\{\frac{2u^2
  m(z)m'(z)}{1-u^2m^2(z)}\Bigg\}
\Im\Bigg\{\frac{2u^2m(w)m'(w)}{1-u^2m^2(w)}\Bigg\}\Bigg|\,\diff E\,
\diff F \\ \leq \|\psi\|_\infty \|\phi\|_\infty \Bigg(
\frac{16\prod_{j=1}^2\arg\big\{1 - u^2
  r_{\eta_j}(\tilde{\theta}_{\eta_j,u})
  \exp(i2\tilde{\theta}_{\eta_j,u})\big\}}{\pi^2} \mathbf{1}_{u <
  \frac{1}{2}} + 18^2 \mathbf{1}_{u\geq \frac{1}{2}}\Bigg),\\
  \leq 18^2\|\psi\|_\infty \|\phi\|_\infty \Big( u^4 \mathbf{1}_{u <
  \frac{1}{2}} + \mathbf{1}_{u\geq \frac{1}{2}}\Big),
\end{multline*}
uniformly in $\eta_j \leq 1$, proving the first part of the Lemma. The
second part holds by applying the first bound and integrating over
$u$.
\end{proof}

With the results in Lemmas \ref{lem:argvar},
\ref{lem:propofcurve}, \ref{lem:integrationbound} and
\ref{lem:domconvinu} in hand, we will be able to prove
Theorem~\ref{thm:genPlemelj}. Our sketch is as follows. Given test
functions $\psi, \phi \in C_b(\R)$ we must compute
\begin{multline*}
\lim_{\eta_1,\eta_2 \downarrow 0} \frac{c}{\pi^2 k_\alpha}
  \int_0^\infty\frac{1}{u^{1+\alpha}}\iint_{\R^2} \psi(E)
  \phi(F)\Bigg[\Im\bigg\{\frac{2u^2m(E+i\eta_1)m'(E+i\eta_1)}{1-u^2m^2(E+i\eta_1)}\bigg\}\\
  \times\Im\bigg\{\frac{2u^2m(F+i\eta_2)m'(F+i\eta_2)}{1-u^2
    m^2(F+i\eta_2)}\bigg\} \Bigg] \,\diff E\,\diff F
\,\diff u, \end{multline*}
because of the bound in Lemma~\ref{lem:domconvinu} we may not only
interchange orders of integration freely by Fubini Theorem, but we may
take the limit in $\eta_1,\eta_2\downarrow 0$ inside the integral in
$u$ by dominated convergence Theorem. It suffices then to compute the
limit
\begin{multline*}
\lim_{\eta_1,\eta_2\downarrow 0}\iint_{\R^2} \psi(E)\phi(F)
\Bigg[\Im\bigg\{\frac{2u^2m(E+i\eta_1)m'(E+i\eta_1)}{1-u^2m^2(E+i\eta_1)}\bigg\}\\
  \times\Im\bigg\{\frac{2u^2m(F+i\eta_2)m'(F+i\eta_2)}{1-u^2
    m^2(F+i\eta_2)}\bigg\} \Bigg]\,\diff E\,\diff F,
\end{multline*}
this is the product of two one-dimensional integrals, for which we will 
apply dominated convergence once more by starting with the useful 
representation of Lemma~\ref{lem:argvar}. To facilitate computation, we 
work with a mollified version of $\psi_t$ of $\psi$ and $\phi_s$ of 
$\phi$ so that we may integrate by parts. This mollification poses no 
issues in the analysis since Lemma~\ref{lem:domconvinu} implies that 
uniformly in $\eta_1,\eta_2 \leq 1$ the difference between these limits 
is bounded by $18^2(2^{\alpha}) (4-\alpha)^{-1} \|\psi_t - 
\psi\|_\infty \|\phi_s - \phi\|_\infty$ which will go to zero when we 
take the appropriate limits in $s$ and $t$.

\begin{proof}[Proof of Theorem~\ref{thm:genPlemelj}]
Following the sketch above,
let $t> 0$ and consider $\psi_t = P_t*\psi$ where $P_t$ is the Poisson
kernel. Recall that $\|\psi_t - \psi\|_\infty$ goes to $0$ in the
limit that $t\to 0$ and $\psi_t$ is infinitely
differentiable, further $\|\psi_t'\|_\infty \leq \|P_t'\|_{1}
\|\psi\|_\infty$ for every $t>0$. Applying the representation in
Lemma~\ref{lem:argvar} to $\psi_t$ we have
\begin{multline*}
\int_{\R} \psi_t(E)
\Im\bigg\{\frac{2u^2m(E+i\eta)m'(E+i\eta)}{1-u^2m^2(E+i\eta_1)}\bigg\}\,\diff E\\
=
-\int_\R \psi_t(x)\frac{\partial}{\partial x}\arg\big\{1 -
u^2m^2(x+i\eta)\big\}\,\diff x
\\
=\int_\R \psi_t'(x) \arg\big\{1 - u^2m^2(x+i\eta)\big\} \,\diff x.
\end{multline*}
For each $u$ we will apply Dominated Convergence Theorem to take the
limit $\eta \downarrow 0$ inside the integral.  To this end note that
it suffices to show
\[
\|\psi\|_\infty \|P_t'\|_1 \sup_{\eta\leq 1} \big|\arg\big\{1 -
u^2m^2(x+i\eta)\big\}\big|,
\]
is integrable on $\R$ for almost every $u>0$.  Observe that
\begin{equation}
\label{eqn:regofargument}
|\zeta| < \frac{1}{2} \implies|\arg\{1 - \zeta\}|\leq |\zeta|.
\end{equation}
Recall from Lemma~\ref{lem:argvar} that $|m^2(x+i\eta)| =
r_\eta^2(\theta)$ where $\theta$ is the argument of $m(z)$. Further,
we have $x = E_\eta(\theta)$. Let
\[
\epsilon_u  :=2^{-\frac{1}{2}} \min(1,u^{-1}) \quad\hbox{and}\quad
\theta^*_u := \arcsin\Bigg(\frac{-\eta \epsilon_u}{1- \epsilon_u^2}\Bigg),
\]
where $\arcsin$ is taken to be the principal value. Note that
$r_\eta(\theta^*_u)=\epsilon_u$ and since
$\theta^*_u \in \big(-\frac{\pi}{2},0\big)$ we have
\[
\begin{split}
 E_\eta(\theta^*_u) &= \bigg( \frac{\epsilon_u^2 +
   1}{\epsilon_u\big(1-\epsilon_u^2\big)}\bigg) \sqrt{\epsilon_u^4
   -(2+ \eta^2)\epsilon_u^2 + 1} \\ &\leq \bigg( \frac{\epsilon_u^2 +
   1}{\epsilon_u\big(1-\epsilon_u^2\big)}\bigg) \sqrt{\epsilon_u^4
   -2\epsilon_u^2 + 1} =: x^*_u,
\end{split}
\]
moreover since $E_\eta(\theta)$ is strictly increasing and is a
bijection from $(-\pi,0)$ to $\R$ it follows that
\[
|x| \geq x^*_u \implies |m^2(x+i\eta)| \leq \epsilon_u^2,
\]
since $x^*_u$ is a constant only depending on $u$, it follows that
uniformly in $\eta>0$ we have
\begin{align*}
|x|\geq \max(x^*_u,3) &\implies \big|\arg\big\{1 -
u^2m^2(x+i\eta)\big\}\big| \leq  u^2
|m^2(x+i\eta)|,\\ |x|\leq\max(x^*_u,3) &\implies
\big|\arg\big\{1 - u^2m^2(x+i\eta)\big\}\big| \leq \pi,
\end{align*}
but observe that $\sup_{\eta \leq 1} |m^2(x+i\eta)|$
is integrable over $|x| \geq \max(x^*_u,3)$. Hence
Dominated Convergence Theorem applies if we prove that the
pointwise limit
\[
\lim_{\eta\downarrow 0} \arg\big\{1 - u^2m^2(x+i\eta)\big\},
\]
exists almost everywhere.  To compute
this limit, we first consider the set of $x$ with $|x|<2$. 
In this region, for each point $x$ there is a small enough $\eta$ 
such that $1-u^2m^2(x+i\eta)$,
is not on the branch cut of $\arg$.  Since $\arg$ is continuous in this 
region we may take the limit as $\eta \downarrow 0$ under the argument.
We have the limits
\[
m^2(x+i \eta) \to \begin{cases}
\frac{x^2 - 2}{2} - i \frac{x\sqrt{4-x^2}}{2}& |x| < 2 \\
\frac{x^2}{2} - 1 - \frac{x\sgn(x)\sqrt{x^2-4}}{2}& |x| \geq 2
\end{cases}
\]
so that
\[
|x| < 2 \implies 
\lim_{\eta\downarrow 0} \arg\big\{1 - u^2m^2(x+i\eta)\big\} =
\arg\bigg\{ 1 + u^2 - \frac{u^2 x^2}{2} + i
\frac{u^2x\sqrt{4-x^2}}{2}\bigg\}
\]
For $|x| \geq 2$, let $\sigma(x)$ be the semicircular density, recall
\[
m(x+i\eta) = \int_{\R} \frac{(x- y) \sigma(\diff y)}{(x-y)^2 + \eta^2} -
i \int_\R \frac{\eta \sigma(\diff y)}{(x-y)^2+\eta^2}
\]
by squaring this integral representation above, we see for $x\geq 2$,
$1-u^2m^2(x+i\eta)$ will be in the upper-half plane and for $x\leq
-2$, $1-u^2m^2(x+i\eta)$ will be in the lower-half plane. The branch
cut of $\arg$ is approached for $|x| > 2$ only for those $x$ for which
\begin{equation}
\label{eqn:limreal}
\lim_{\eta \downarrow 0}\Re \{1 - u^2m^2(x+i\eta)\} \leq 0,
\end{equation}
the limiting region for which this holds is
\[
1 \leq u^2 m^2_+(x),
\]
corresponding to
\[
2 \leq |x| \leq u + \frac{1}{u},
\]
which is non-empty only for $u\geq 1$. Hence  we have the following limit
\[
|x|\geq 2 \implies
\lim_{\eta\downarrow 0} \arg\big\{1 - u^2m^2(x+i\eta)\big\} =
\pi \bigg(\mathbf{1}_{2\leq x \leq u +u^{-1}} - \mathbf{1}_{
  -u-u^{-1}\leq x\leq -2} \bigg)\mathbf{1}_{u\geq 1},
\]
holding for almost every $u>0$. Combining these results
\begin{multline*}
\lim_{\eta\downarrow 0} \int_\R \psi_t'(x) \arg\big\{1 -
u^2m^2(x+i\eta)\big\} \,\diff x =\\ \int_\R \psi_t'(x) \arg\bigg\{ 1 +
u^2 - \frac{u^2 x^2}{2} + i
\frac{u^2x\sqrt{4-x^2}}{2}\bigg\}\mathbf{1}_{|x|\leq 2}\,\diff x +
\\\pi \Big(\psi_t\big(u+u^{-1}\big) - \psi_t(2)\Big)\mathbf{1}_{u\geq
  1} - \pi \Big(\psi_t(-2) -
\psi_t\big(-u-u^{-1}\big)\Big)\mathbf{1}_{u \geq 1},
\end{multline*}
for almost every $u>0$. Integrating the first term by parts yields
\begin{multline*}
\int_\R \psi_t'(x) \arg\bigg\{ 1 + u^2 - \frac{u^2 x^2}{2} + i
\frac{u^2x\sqrt{4-x^2}}{2}\bigg\}\mathbf{1}_{|x|\leq 2}\,\diff x
=\\ \int_\R \psi_t(x)\frac{\diff}{\diff x}\arg\bigg\{ 1 +
u^2 - \frac{u^2 x^2}{2} + i
\frac{u^2x\sqrt{4-x^2}}{2}\bigg\}\mathbf{1}_{|x|\leq 2}\,\diff x \\
	+ \pi \Big\{\psi_t(2) + \psi_t(-2)\Big\}\mathbf{1}_{u \geq 1},
\end{multline*}
holding for almost every point $u > 0$.
Note that for $|x|< 2$,
\[
\frac{\diff}{\diff x}\arg\bigg\{ 1 + u^2 - \frac{u^2 x^2}{2} + i
\frac{u^2x\sqrt{4-x^2}}{2}\bigg\} = \frac{2u^4+2u^2 - u^2x^2}{
  \sqrt{4-x^2} \Big( -u^2x^2 + u^4 +2u^2+1\Big)}
\]
which is absolutely integrable in $[-2,2]$ so that
\begin{multline*}
\lim_{\eta_1 \downarrow 0 } \int_\R \psi_t(E) \Im\bigg\{\frac{2u^2 m(E+i\eta_1)m'(E+i\eta_1)}{1-u^2m^2(E+i\eta_1)}\bigg\}\,\diff E\\
  =\int_\R \psi_t(x) \frac{2u^4+2u^2 - u^2x^2}{ \sqrt{4-x^2} \Big(
    -u^2x^2 + u^4 +2u^2+1\Big)}\mathbf{1}_{|x|\leq 2}\,\diff x\\ + \pi
  \Big\{\psi_t\big(u+u^{-1}\big)
  +\psi_t\big(-u-u^{-1}\big)\Big\}\mathbf{1}_{u\geq 1}.
\end{multline*}
Using this argument with the original integral (with $\phi_s =
P_s*\phi$) along with the previous Lemmas allowing us to pull the
limit as $\eta_1,\eta_2 \downarrow 0$ under the integral over $u$
which implies the limit
\begin{multline*}
  \lim_{\eta_1,\eta_2 \downarrow 0} \frac{c}{\pi^2 k_\alpha}
  \int_0^\infty\frac{1}{u^{1+\alpha}}\iint_{\R^2} \psi_t(E)
  \phi_s(F)\Bigg[\Im\bigg\{\frac{2u^2m(E+i\eta_1)m'(E+i\eta_1)}{1-u^2m^2(E+i\eta_1)}\bigg\}\\
  \times\Im\bigg\{\frac{2u^2m(F+i\eta_2)m'(F+i\eta_2)}{1-u^2
    m^2(F+i\eta_2)}\bigg\} \Bigg]\,\diff E\,\diff F \,\diff u =
  \frac{c}{k_\alpha}\int_0^\infty
  \frac{\Lambda_u(\psi_t)\Lambda_u(\phi_s)}{u^{1+\alpha}} \,\diff u.
\end{multline*}
Using the bound in Lemma~\ref{lem:domconvinu} allows us to take the 
limit as $t$ and $s$ go down to $0$ so that the above formula holds for 
$\psi$ and $\phi$ as well (we omit the argument that 
$u^{-1-\alpha}\Lambda_u(\psi_t)\Lambda_u(\phi_s)$ converges in 
$L^1(\R)$ as $t,s\downarrow 0$ since it is straightforward). To prove 
the final statement of the Theorem it suffices to show, for 
$z \in \C\backslash \R$,
\begin{equation}
\label{eqn:verify}
\Lambda_u\bigg(\frac{1}{x-z}\bigg) = \frac{2u^2 m(z)m'(z)}{1-u^2m^2(z)},
\end{equation}
note that the right hand side, for $u < 1$ is analytic on 
$\C \backslash [-2,2]$ while for $u\geq 1$ is meromorphic with simple 
poles at $\pm(u + u^{-1})$, that is, the function
\begin{equation}
\label{eqn:cauchyproblem1}
\Phi(z) := 
\frac{2u^2 m(z)m'(z)}{1-u^2m^2(z)} + \bigg(\frac{1}{z-u-u^{-1}} + \frac{1}{z+u+u^{-1}}\bigg)
\mathbf{1}_{u\geq 1},
\end{equation}
is analytic on $\C \backslash [-2,2]$. Let 
\[
\Phi_{\pm} (x)  = \lim_{\eta \downarrow 0} \Phi(x \pm i \eta),
\]
then by the well-known Sokhotski-Plemelj relation \cite[Section 34.2]{Ga90}
\begin{equation}
\label{eqn:cauchyproblem2}
\frac{1}{2\pi i}\int_{\R} \frac{\Phi_+(x) - \Phi_-(x)}{x-z} \,\diff x=
\Phi(z),
\end{equation}
given our computation above, 
\begin{equation}
\label{eqn:cauchyproblem3}
\frac{\Phi_+(x) - \Phi_-(x)}{2\pi i} = \frac{2u^4+ 2u^2 - u^2x^2}
{\pi\sqrt{4-x^2}\big(-u^2x^2+ u^4 + 2u^2 +1\big)} \mathbf{1}_{|x|\leq 2}
,
\end{equation}
inserting equation~\eqref{eqn:cauchyproblem3} and 
equation~\eqref{eqn:cauchyproblem1} into 
equation~\eqref{eqn:cauchyproblem2} gives the desired 
equation~\eqref{eqn:verify} upon rearranging terms.
\end{proof}

\section{Proof of Corollary~\ref{t:plemelj}}
\label{s:corPlem}
\begin{proof}
Using the representation of Theorem~\ref{thm:genPlemelj}, we may write 
the covariance kernel as the sum 
\begin{equation}
\label{eqn:exchangeduint}
\langle \Ker_\alpha, \psi\otimes \phi \rangle = T_1 + T_2 + T_3 + T_4,
\end{equation}
where we have defined
\begin{equation}
\label{eqn:labelledterms}
\begin{split}
T_1 &:= 
\frac{c}{k_\alpha} \int_{-2}^2\int_{-2}^2 \int_0^\infty
\frac{u^{3-\alpha}(2u^2+2 -E^2)(2u^2+2-F^2)\,\diff
u}{\pi^2(-u^2E^2+(u^2+1)^2)(-u^2F^2+(u^2+1)^2)}
\frac{\psi(E)\phi(F)\,\diff E \,\diff F}{\sqrt{4-E^2}\sqrt{4-F^2}},\\
T_2 &:= \frac{c}{k_\alpha}\int_{-2}^2\frac{\psi(E)}{\sqrt{4-E^2}}\int_1^\infty
\frac{u^{1-\alpha}(2u^2+2-E^2)\big(\phi(u+u^{-1})
+\phi(-u-u^{-1})\big)}{\pi(-u^2E^2+(u^2+1)^2)} \,\diff u,\\
T_3 &:= \frac{c}{k_\alpha}\int_{-2}^2\frac{\phi(F)}{\sqrt{4-F^2}}\int_1^\infty
\frac{u^{1-\alpha}(2u^2+2-F^2)\big(\psi(u+u^{-1}) +
\psi(-u-u^{-1})\big)}{\pi(-u^2F^2+(u^2+1)^2)} \,\diff u,\\
T_4&:=\frac{c}{k_\alpha}\int_{1}^\infty\frac{\big(\psi(u+u^{-1})
+\psi(-u-u^{-1})\big)\big(\phi(u+u^{-1}) +
\phi(-u-u^{-1})\big)}{u^{1+\alpha}}\,\diff u.
\end{split}
\end{equation}
note that we have freely interchanged orders of integration to obtain the
equality~\eqref{eqn:exchangeduint}.
We will rewrite each of the integrals 
in~\eqref{eqn:labelledterms} starting with the term $T_4$ 
in~\eqref{eqn:labelledterms}. Consider the change of variables
\[
u = \frac{1}{m_+(s)}\qquad s \in [2,\infty),
\]
which is well-defined since for $s \in[2,\infty)$, $m_+$ is a
decreasing function and takes values in $(0,1]$.  Using the fixed
point equation satisfied by $m_+$, we change variables for $T_4$ and
obtain
\begin{multline}
\label{eqn:term4manip}
T_4 = \frac{c}{k_\alpha}\int_{2}^\infty\big(\psi(E)
+\psi(-E)\big)\big(\phi(E) + \phi(-E)\big)m_+(E)^{\alpha-1}
m_+'(E)\,\diff E\\ = \frac{c}{k_\alpha} \iint_{(\R\backslash[-2,2])^2}
\big(\delta_{E-F} + \delta_{E+F}\big)|m_+(E)|^{\alpha-1} |m'_+(E)|
\psi(E)\phi(F)\,\diff E \,\diff F.
\end{multline}
Next, consider the term $T_3$ in~\eqref{eqn:labelledterms} (with the 
term $T_2$ being identical). We once more change variables  to $u = 
\frac{1}{m_+(E)}$ to rewrite this integral as
\begin{multline}
\label{eqn:term3manip}
T_3 = \frac{c}{\pi k_\alpha}\bigg(\int_{-2}^2\int_2^\infty 
\frac{|m_+(E)|^{\alpha-1}(2+m_+(E)^2(2-F^2))m_+'(E)}{(-m_+(E)^{2} 
F^2+(m_+(E)^{2}+1)^2)\sqrt{4-F^2}}\psi(E)\phi(F) \,\diff F\,\diff E \\ 
+\int_{-2}^2\int_{-\infty}^{-2} 
\frac{|m_+(E)|^{\alpha-1}(2+m_+(E)^2(2-F^2))m_+'(E)}{(-m_+(E)^{2} 
F^2+(m_+(E)^{2}+1)^2)\sqrt{4-F^2}}\psi(E)\phi(F) \,\diff F\,\diff 
E\bigg),
\end{multline}
For the remaining term, $T_1$, in~\eqref{eqn:exchangeduint} we will 
compute for almost every $E,F \in(-2,2)$ the integral
\begin{equation}
\label{eqn:residueforbox}
\frac{c}{k_\alpha\sqrt{4-E^2}\sqrt{4-F^2}}\int_0^\infty 
\frac{u^{3-\alpha}(2u^2+2 -E^2)(2u^2+2-F^2)\,\diff 
u}{\pi^2(-u^2E^2+(u^2+1)^2)(-u^2F^2+(u^2+1)^2)}.
\end{equation}
the above integral is simply the pointwise limit
\begin{multline*}
- \lim_{\eta_1, \eta_2\downarrow 0}\frac{1}{4\pi^2} \bigg\{ 
C(E+i\eta_1,F+i\eta_2) + C(E-i\eta_1,F-i\eta_2) \\- 
C(E-i\eta_1,F+i\eta_2) - C(E+i\eta_1,F-i\eta_2)\bigg\},
\end{multline*}
which is easy to evaluate given the explicit formula in 
Remark~\ref{rem:branchcut} as
\begin{multline}
\label{eqn:insidebox}
\frac{-c }{\pi k_\alpha \sin\big(\frac{\pi \alpha}{2}\big)}
\times\\
\bigg\{
\frac{m_+(E)m'_+(E)m_+(F)m'_+(F)}{ m^2_+(E) - m^2_+(F)}
\bigg[\Big(-m_+(E)^2\Big)^{\frac{\alpha}{2} - 1} 
- \big(-m_+(F)^2\big)^{\frac{\alpha}{2} -1 }\bigg]
\\
+\frac{m_-(E)m'_-(E)m_-(F)m'_-(F)}{ m^2_-(E) - m^2_-(F)}
\bigg[\Big(-m_-(E)^2\Big)^{\frac{\alpha}{2} - 1} 
- \big( -m_-(F)^2\big)^{\frac{\alpha}{2} - 1}
\bigg]
\\
-\frac{m_+(E)m'_+(E)m_-(F)m'_-(F)}{ m^2_+(E) - m^2_-(F)}
\bigg[\Big(-m_+(E)^2\Big)^{\frac{\alpha}{2} - 1} 
- \Big(-m_-(F)^2\Big)^{\frac{\alpha}{2} - 1}
\bigg]
\\
-\frac{m_-(E)m'_-(E)m_+(F)m'_+(F)}{ m^2_-(E) - m^2_+(F)}
\bigg[\Big(-m_-(E)^2\Big)^{\frac{\alpha}{2} - 1} 
-  \Big(-m_+(F)^2\Big)^{\frac{\alpha}{2} - 1}\bigg]
\bigg\}.
\end{multline}
Using the formulas~\eqref{eqn:term4manip}, \eqref{eqn:term3manip} and
~\eqref{eqn:insidebox} for $T_1$, $T_2$, $T_3$ and 
$T_4$ in the original equations~\eqref{eqn:exchangeduint} 
and~\eqref{eqn:labelledterms}, proves the first part of the Corollary. 
We now describe the simplifications yielding the second part of the 
Corollary pertaining to the case $\alpha=3$. When $\alpha = 3$, 
equation~\eqref{eqn:insidebox} further simplifies to
\begin{multline}
\label{eqn:alpha3insidebox}
\frac{c }{\pi k_3}
\bigg\{
\frac{im_+(E)m'_+(E)m_+(F)m'_+(F)}{ m_+(E) + m_+(F)}
-\frac{i m_-(E)m'_-(E)m_-(F)m'_-(F)}{ m_-(E) + m_-(F)}
\\
-\frac{im_+(E)m'_+(E)m_-(F)m'_-(F)}{ m_+(E) - m_-(F)}
+\frac{im_-(E)m'_-(E)m_+(F)m'_+(F)}{ m_-(E) - m_+(F)}
\bigg\} \\
=-\frac{2c }{\pi k_3}
\Im\bigg\{
\frac{m_+(E)m'_+(E)m_+(F)m'_+(F)}{ m_+(E) + m_+(F)}
+\frac{m_-(E)m'_-(E)m_+(F)m'_+(F)}{ m_-(E) - m_+(F)}
\bigg\},
\end{multline}
to compute this imaginary part we combine the terms inside the brackets
\begin{multline*}
m_+(F)m_+'(F) \times \\
\frac{m_+(E)m'_+(E)(m_-(E) - m_+(F)) + 
m_-(E)m'_-(E)(m_+(E)+m_+(F))}
{(m_+(E) + m_+(F))( m_-(E) - m_+(F))},
\end{multline*}
note that $|m_\pm(E)|^2 = 1$ so that
\begin{equation}
\label{eqn:alpha3ratio}
m_+(F)m_+'(F)
\frac{2\Re\{m'_+(E)\}+ 2i m_+(F)\Im\{ m_-(E)m'_-(E)\}} 
{1 + im_+(F)\sqrt{4-E^2}-m_+(F)^2},
\end{equation}
note that
\[
2\Re\{m'_+(E)\}+ 2i m_+(F)\Im\{ m_-(E)m'_-(E)\} = 1 + im_+(F) \frac{2-E^2}{\sqrt{4-E^2}}
\]
multiplying this by the conjugate of the denominator of 
equation~\eqref{eqn:alpha3ratio} gives
\begin{multline*}
\bigg(1 + im_+(F) \frac{2-E^2}{\sqrt{4-E^2}}\bigg)
\big(1 - im_-(F)\sqrt{4-E^2}-m_-(F)^2\big) \\
=1 - im_-(F)\sqrt{4-E^2}-m_-(F)^2 -i (m_+(F)-m_-(F))\frac{2-E^2}{\sqrt{4-E^2}}+ (2-E^2),\\
=1 - im_-(F)\sqrt{4-E^2}-m_-(F)^2 + \big(\sqrt{4-F^2} + \sqrt{4-E^2}\big)\frac{2-E^2}{\sqrt{4-E^2}},
\end{multline*}
multiplying this by the remaining factor of $m_+(F)m_+'(F)$ in 
equation~\eqref{eqn:alpha3ratio} gives
\[
-im_+'(F)\big(\sqrt{4-F^2}+\sqrt{4-E^2}\big) 
+m_+(F)m_+'(F)\big(\sqrt{4-F^2} 
+ \sqrt{4-E^2}\big)\frac{2-E^2}{\sqrt{4-E^2}}
\]
the imaginary part of this expression is
\[
-\bigg(\frac{\sqrt{4-F^2}+\sqrt{4-E^2}}{2}\bigg)\bigg(
1 + \frac{(2-F^2)(2-E^2)}{\sqrt{4-F^2}\sqrt{4-E^2}}\bigg)
\]
the modulus square of the denominator of equation~\eqref{eqn:alpha3ratio} can
similarly be verified to equal
\[
\big(\sqrt{4-E^2} + \sqrt{4-F^2}\big)^2,
\]
hence the imaginary part of equation~\eqref{eqn:alpha3ratio} inserted
into equation~\eqref{eqn:alpha3insidebox} yields 
\[
\frac{15c}{8\sqrt{\pi} \pi}
\bigg(\frac{1}{\sqrt{4-F^2}+\sqrt{4-E^2}}\bigg)\bigg(
1 + \frac{(2-F^2)(2-E^2)}{\sqrt{4-F^2}\sqrt{4-E^2}}\bigg),
\]
as claimed.
\end{proof}

\appendix

\section{Integral Identities}
\begin{lemma}
\label{lem:pvintegral1}
Let $\alpha \in(2,4)$ and let $\sigma \in \C$ satisfy $\Re(\sigma) < 0$.
Then
\begin{equation*}
\int_0^\infty \frac{\exp(r\sigma)-r\sigma-1}{r^{\frac{\alpha}{2}+1}}
\,\diff r = k_\alpha (-\sigma)^{\frac{\alpha}{2}}
\end{equation*}
where
\begin{equation*}
k_\alpha = 
\frac{\Gamma\big(2-\frac{\alpha}{2}\big)}{\frac{\alpha}{2}\big( 
\frac{\alpha}{2} - 1\big)},
\end{equation*}
and the principal branch of the power function is taken.
\end{lemma}
\begin{proof}
Starting from,
\begin{equation*}
\int_0^1 (1 - \nu) \frac{\diff^2}{\diff\nu^2} \exp(r \nu \sigma)\,
\diff\nu = \exp(r\sigma)-r\sigma-1,
\end{equation*}
we have
\begin{equation*}
\int_0^\infty \frac{\exp(r\sigma)  - r\sigma - 1}{r^{\frac{\alpha}{2} +
1}} \,\diff r=\int_0^\infty \int_0^1 r^{1 - \frac{\alpha}{2}} \sigma^2
(1 - \nu) \exp(r \nu \sigma)\,\diff \nu\,\diff r,
\end{equation*}
interchanging orders of integration yields
\begin{equation*}
\int_0^1\int_0^\infty  (1 - \nu) \sigma^2  r^{1- \frac{\alpha}{2}}
\exp(r\nu\sigma)\,\diff r\,\diff \nu,
\end{equation*}
changing variables in the integrand for $r$ yields
\begin{equation*}
\sigma^2|\sigma|^{\frac{\alpha}{2} - 2}\int_0^1 (1 -
\nu)\nu^{\frac{\alpha}{2} - 2} \,\diff \nu \int_0^\infty r^{1 -
\frac{\alpha}{2}}\exp( \omega r) \,\diff r \quad\text{where}\quad
\omega := \frac{\sigma}{|\sigma|}.
\end{equation*}
The integral over $\nu$ is  the Beta integral:
\begin{equation*}
\int_0^1 \nu^{\frac{\alpha}{2} - 2} ( 1 - \nu) \, \diff \nu =
B\bigg(\frac{\alpha}{2}-1,2\bigg) = \frac{\Gamma(\frac{\alpha}{2} - 1)
\Gamma(2)}{\Gamma(1 + \frac{\alpha}{2})}
\end{equation*}
while the integral
\begin{equation*}
\int_0^\infty r^{1 - \frac{\alpha}{2}} \exp(\omega r) \,\diff r,
\end{equation*}
will be computed using a Cauchy Integral argument. Since $\Re(\sigma)
< 0$, this implies $- \overline{\omega}$ lies in the right-half plane
$\{z\in \C : \Re(z) > 0\}$, in the sector defined by the positive real
line and the ray along $- \overline{\omega}$, the function
\begin{equation*}
f(z) = z^{1 - \frac{\alpha}{2}} \exp(\omega z),
\end{equation*}
with principal branch cut selected for the power $z^{1-\frac{\alpha}{2}}$,
is holomorphic with an integrable singularity at the origin, so
\begin{align*}
\int_0^\infty r^{1- \frac{\alpha}{2}} \exp(\omega r) \,\diff r &=
-\overline{\omega} ( - \overline{\omega})^{1-\frac{\alpha}{2}}
\int_0^\infty t^{1-\frac{\alpha}{2}} \exp(-t) \,\diff t, \\
&=
-\overline{\omega} ( -
\overline{\omega})^{1-\frac{\alpha}{2}}\Gamma\bigg(2 -
\frac{\alpha}{2}\bigg),
\end{align*}
note that since $\omega$ is on the unit circle, away from
the principal branch cut, and $\frac{\alpha}{2}-1 \in (0,1)$,
\[
\big(-\overline{\omega}\big)^{1-\frac{\alpha}{2}} = 
\big(-\omega\big)^{\frac{\alpha}{2}-1}, 
\]
which implies
\begin{equation*}
\sigma^2 |\sigma|^{\frac{\alpha}{2} - 2} ( - \overline{\omega}) \big(
-\overline{\omega}\big)^{1 - \frac{\alpha}{2}} = - \sigma \big( -
\sigma\big)^{\frac{\alpha}{2} - 1} = (-\sigma)^{\frac{\alpha}{2}}
\end{equation*}
where in the last equality, we used that $\Re(\sigma) < 0$ ,
$\frac{\alpha}{2} -1 \in (0,1)$ (so that we do not cross branch cuts
while combining powers). To obtain the formula for the constant note
\[
\frac{\Gamma\big(\frac{\alpha}{2} - 1 \big)\Gamma(2) 
\Gamma\big(2-\frac{\alpha}{2}\big)}{\Gamma\big(1+\frac{\alpha}{2}\big)} 
= \frac{\Gamma\big(2-\frac{\alpha}{2}\big)}{\frac{\alpha}{2}\big( 
\frac{\alpha}{2} - 1\big)} = k_\alpha.\qedhere
\]
\end{proof}

\begin{lemma}
\label{lem:resIntegral}
Let $\alpha \in (2,4)$ and let $\sigma_1$, $\sigma_2 \in \C \backslash
\R_-$. Then
\begin{equation*}
\int_0^\infty
\frac{r^{\frac{\alpha}{2}-1}}{(r-\sigma_1)(r-\sigma_2)}\,\diff r
=\frac{\pi}{\sin\big(\frac{\pi \alpha}{2}\big)(\sigma_1 - \sigma_2)}
\bigg( (-\sigma_1)^{\frac{\alpha}{2} - 1} - (-
\sigma_2)^{\frac{\alpha}{2} - 1}\bigg)
\end{equation*}
where the principal branch cut is taken for the power function.
When $\sigma_1 = \sigma_2 = \sigma \in \C\backslash\R_-$, then
the integral equals
\begin{equation*}
  \int_0^\infty \frac{r^{\frac{\alpha}{2} - 1}}{(r-\sigma)^2} \,\diff
  r = \frac{\pi(\frac{\alpha}{2} - 1)(-\sigma)^{\frac{\alpha}{2}
      -2}}{\sin\big(\frac{\pi\alpha}{2}\big)}.
\end{equation*}
\end{lemma}

\begin{proof}
Let
\begin{equation*}
  f(z): = \frac{z^{\frac{\alpha}{2}-1}}{(z + \sigma_1)(z + \sigma_2)},
\end{equation*}
if we integrate this function on the counterclockwise keyhole contour
that avoids the branch cut on the negative reals then the residue
theorem yields
\begin{multline*}
\bigg(\exp\bigg[i \pi \bigg( \frac{\alpha}{2} - 1\bigg)\bigg] -
\exp\bigg[-i \pi \bigg(\frac{\alpha}{2} - 1\bigg)\bigg]\bigg)
\int_0^\infty \frac{r^{\frac{\alpha}{2}-1}}{(r -
  \sigma_1)(r-\sigma_2)}\,\diff r \\ =\frac{2 \pi i}{\sigma_2 -
  \sigma_1} \Big((-\sigma_1)^{\frac{\alpha}{2} - 1} -
(-\sigma_2)^{\frac{\alpha}{2} -1 }\Big)
\end{multline*}
which when rearranged yields
\begin{equation*}
  \int_0^\infty \frac{r^{\frac{\alpha}{2}-1}}{(r -
    \sigma_1)(r-\sigma_2)}\,\diff r=  \frac{\pi}{\sin\big(\frac{\pi
        \alpha}{2}\big)(\sigma_1 - \sigma_2)}
  \Big((-\sigma_1)^{\frac{\alpha}{2} - 1} -
  (-\sigma_2)^{\frac{\alpha}{2} -1 }\Big)
\end{equation*}
as required. Setting $\sigma_2 = \sigma$ and letting $\sigma_1 \to \sigma_2$
yields the second integral.
\end{proof}

\bibliographystyle{alpha}

\end{document}